\documentclass[12pt,reqno]{amsart}
\usepackage{scrtime}
\title[Non-backtracking spectrum of universal covers \today]
      {The non-backtracking spectrum of the universal cover of a graph}

\date{November 2007}

\author{Omer Angel}
\address[Omer Angel]{Math\\UofT}
\email{angel@math.utoronto.ca}

\author{Joel Friedman}
\address[Joel Friedman]{CS\\UBC}
\email{jf@cs.ubc.ca}

\author{Shlomo Hoory}
\address[Shlomo Hoory]{Haifa Research Lab.\\IBM}
\email{shlomoh@il.ibm.com}

\thanks{This research was done while OA and SH were at UBC;
  OA held a PIMS postdoctoral fellowship.}

\newcommand{\finiteRS}{{finite ratio system}}
\newcommand{\generalizedRS}{{ratio system}}
\usepackage{amsmath,amsthm,amssymb,amsfonts}
\usepackage{graphicx, pst-all}

\usepackage{ifpdf}
\ifpdf
  \usepackage[pdftex,colorlinks,linkcolor=blue,citecolor=blue]{hyperref}
\else
  \usepackage[colorlinks,linkcolor=blue,citecolor=red]{hyperref}
\fi

\newtheorem{thm}{Theorem}[section]
\newtheorem{lemma}[thm]{Lemma}
\newtheorem{claim}[thm]{Claim}
\newtheorem{prop}[thm]{Proposition}
\newtheorem{coro}[thm]{Corollary}
\newtheorem{note}[thm]{Note}

\theoremstyle{definition}
\newtheorem{defn}[thm]{Definition}

\newcommand{\clmref}[1]{Claim~\ref{C:#1}}

\newcommand{\corref}[1]{Corollary~\ref{C:#1}}

\newcommand{\lemref}[1]{Lemma~\ref{L:#1}}

\newcommand{\thmref}[1]{Theorem~\ref{T:#1}}

\newcommand{\romenum}{\renewcommand{\labelenumi}{(\roman{enumi})}}

\newcommand{\from}{\colon}

\newcommand{\cG}{{\mathcal G}}

\newcommand{\eps}{\varepsilon}
\newcommand{\lam}{\lambda}

\newcommand{\R}{\mathbb{R}}
\newcommand{\C}{\mathbb{C}}
\newcommand{\tA}{\widetilde A}
\newcommand{\tB}{\widetilde B}
\newcommand{\tE}{\widetilde E}
\newcommand{\tG}{\widetilde G}
\newcommand{\tH}{\widetilde H}
\newcommand{\tM}{\widetilde M}
\newcommand{\tV}{\widetilde V}

\newcommand{\gr}{\text{\rm gr}}

\newcommand{\dmin}{{d_{\rm min}}}
\newcommand{\dmax}{{d_{\rm max}}}
\newcommand{\relker}[1]{\ker_{#1} M}
\newcommand{\relnullker}[1]{\ker_{#1}^0 M}

\newcommand{\Tzw}{\begin{psmatrix}[colsep=3mm,rowsep=0mm]
[name=z] z \\[0mm]
[name=w] w
\ncline{z}{w}
\end{psmatrix}}

\newcommand{\Tuxw}{\begin{psmatrix}[colsep=3mm,rowsep=0mm]
[name=u] u & [name=x] x & [name=w] w
\ncline{x}{u} \ncline{x}{w}
\end{psmatrix}}

\newcommand{\Tuxwv}{\begin{psmatrix}[colsep=3mm,rowsep=0mm]
[name=u] u & [name=x] x & [name=w] w \\[0mm]
           & [name=v] v &
\ncline{x}{u} \ncline{x}{v} \ncline{x}{w}
\end{psmatrix}}

\newcommand{\Tuxzw}{\begin{psmatrix}[colsep=2mm,rowsep=0mm]
           & [name=z] z &            \\[0mm]
[name=x] x &            & [name=w] w \\[0mm]
[name=u] u &            &
\ncline{x}{u} \ncline{x}{z} \ncline{z}{w}
\end{psmatrix}}

\begin{document}

\begin{abstract}
  A non-backtracking walk on a graph, $H$, is a directed path of directed
  edges of $H$ such that no edge is the inverse of its preceding edge.
  Non-backtracking walks of a given length can be counted using the
  non-backtracking adjacency matrix, $B$, indexed by $H$'s directed edges
  and related to Ihara's Zeta function.

  We show how to determine $B$'s spectrum in the case where $H$ is a tree
  covering a finite graph. We show that when $H$ is not regular, this
  spectrum can have positive measure in the complex plane, unlike the
  regular case. We show that outside of $B$'s spectrum, the corresponding
  Green function has ``periodic decay ratios.'' The existence of such a
  ``ratio system'' can be effectively checked, and is equivalent to being
  outside the spectrum.

  We also prove that the spectral radius of the non-backtracking walk
  operator on the tree covering a finite graph is exactly $\sqrt\gr$, where
  $\gr$ is the growth rate of the tree. This further motivates the
  definition of the graph theoretical Riemann hypothesis proposed by Stark
  and Terras~\cite{ST}.

  Finally, we give experimental evidence that for a fixed, finite graph,
  $H$, a random lift of large degree has non-backtracking new spectrum near
  that of $H$'s universal cover. This suggests a new generalization of
  Alon's second eigenvalue conjecture.
\end{abstract}

\maketitle

%\tableofcontents

%%%%%%%%%%%%%%%%%%%%%%%%%%%%%%%%%%%%%%%%%%%%%%%%%%%%%%%%%%%%
\section{Introduction}
%%%%%%%%%%%%%%%%%%%%%%%%%%%%%%%%%%%%%%%%%%%%%%%%%%%%%%%%%%%%

Let $H$ be a finite, simple (i.e., without self-loops or multiple
edges), connected graph. 
The main result of this paper is to describe the spectrum of operators
on the universal cover, $\tH$, of $H$, that are ``lifts'' from
operators on $H$.  Let us briefly describe this result, 
saving precise definitions and statements for later sections.

A {\em cover} of $H$ is a graph homomorphism $\pi\from G\to H$ such that
$\pi$ induces a bijection between the edges incident on $x$ and those
incident on $\pi(x)$. Note that $G$ may be infinite. The {\em universal
  cover} $\tH$ of $H$ is the unique (up to isomorphism) cover of $H$ that
is a tree. A cover of $H$ may be thought of as a graph with the same local
structure as $H$. The universal cover is in several senses an approximation
to typical large covers of $H$.

A {\em connected local operator} on $H$ is a matrix $M$ indexed by $H$'s
vertices, such that if $x\ne y$ then $M_{x,y}\neq0$ iff $x\sim y$,
where $x\sim y$ means that $x$ is adjacent to $y$; $M$ may be viewed as
an operator, in that
for a function $f$ on $H$
we have a function $Mf$ defined by
$(Mf)_x = \sum_{y\sim x} M_{x,y} f_y$. Using the bijection on
neighbourhoods induced by $\pi$, any such operator extends in a canonical
way to a bounded
operator $\tM$ on $\ell^2(\tH)$.  (See Section~\ref{sec:prelim} for
more detailed definitions.) One of our main results is an algebraic
description of the spectrum of $\tM$.

\begin{thm}\label{T:main_rough}
  Fix a finite graph, $H$, and symmetric, connected, local operator, $M$,
  on $H$. There is an explicit (in terms of $H,\lam$) set, $S$, of
  polynomial equations in variables $\{r_e\}_{e\in E(H)}$ taking values in
  $\C\cup\{\infty\}$, such that $\tM-\lam I$ is invertible iff $S$ has a
  solution with norm $\alpha(r)<1$, where $\alpha(r)$ is the largest
  eigenvalue of a matrix whose entries are given via $H$ as polynomials in
  the $\{r_e\}$.
\end{thm}

See \thmref{main} below for a statement including the algebraic
representation. This result gives an algorithm for finding the spectrum of
of $\tH$, and shows that it is determined by some algebraic curves.

A significant and motivating case of \thmref{main_rough}
is the non-backtracking spectrum of a
graph. A non-backtracking walk on a graph, $H$, is a directed path in $H$
such that no edge is the inverse of the preceding edge. The
non-backtracking matrix of the graph allows counting such paths in the same
way that the adjacency matrix of a graph is used to count paths in the
graph. Non-backtracking walks appear in numerous applications and
estimating their number is a fundamental problem. The matrix $B$ is indexed
by directed edges of $H$, with $B_{e,f}=1$ iff $(e,f)$ form a
non-backtracking path of length 2. This may be seen as an asymmetric, local
operator on the edge graph of $H$, and in this way extended to an operator
on $\tH$. This operator is asymmetric, but we can still use \thmref{main}
to characterize $\tB$'s spectrum. Note that while $B$ and $\tB$ are real
operators, they are not self-adjoint and so their spectra may be complex.

\begin{thm}\label{T:Bspec_rough}
  Given $H,\lam$, there is an explicit set of polynomial equations in
  variables $\{r_e\}_{e\in E(H)}$ taking values in $\C\cup\{\infty\}$, such
  that $\tB-\lam I$ is invertible iff there is a solution with norm
  $\alpha(r)<1$, where $\alpha(r)$ is explicitly given in terms of $H$.
\end{thm}

As an example, we consider the case where $H$ is $K_4$ with one edge
deleted. Already in that case, the spectrum of the non-backtracking path
operator on $\tB$ has a non-trivial structure, see
Figure~\ref{fig-tree-spec}. We are also interested in the relation between
the non-backtracking spectrum of $\tH$ and that of a typical large but
finite cover of $H$. Figure~\ref{fig-lift-spec} shows the spectrum of a
randomly chosen cover of this $H$. This and other numerical experiments
give strong indication that the spectra converge to the spectrum of the
universal cover (determined rigorously in this paper.)

\begin{figure}[ht]
  \begin{center}
    \includegraphics[width=.7\textwidth]{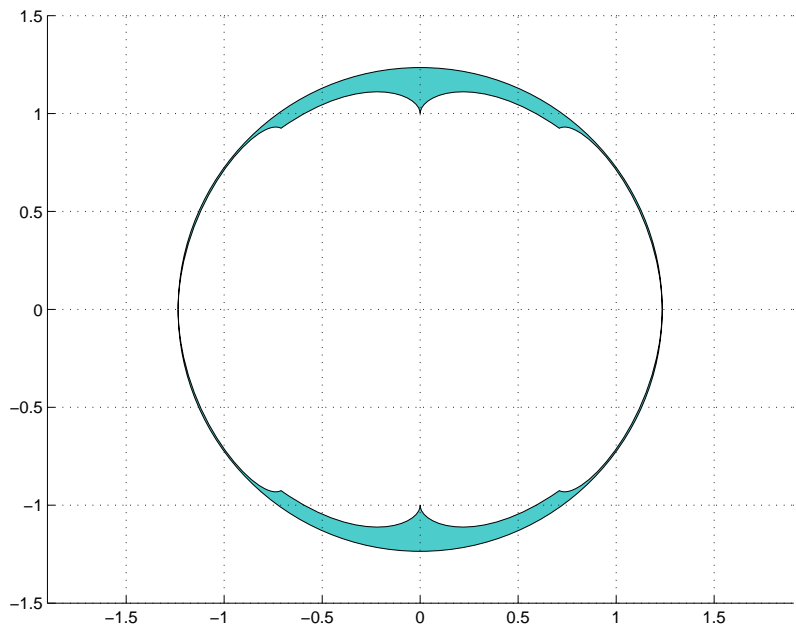}
    \caption{The non-backtracking spectrum of the universal cover of $K_4$
      minus an edge.}
    \label{fig-tree-spec}
  \end{center}
\end{figure}

\begin{figure}[ht]
  \begin{center}
    \includegraphics[width=\textwidth]{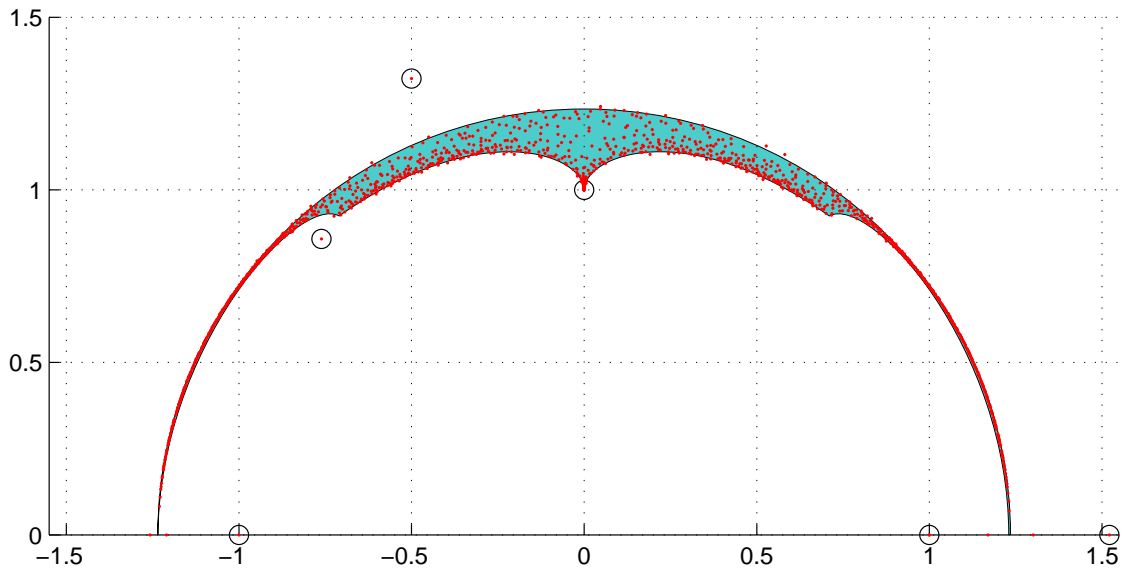}
    \caption{The non-backtracking spectrum of a random 300-lift of $K_4$
      minus an edge. The shaded area is the non-backtracking spectrum of
      its universal cover; the circled points are the spectrum for the base
      graph.}
    \label{fig-lift-spec}
  \end{center}
\end{figure}

%%%%%%%%%%%%%%%%%%%%%%%%%%%%%%%%%%%%%%%%%%%%%%%%%%%%%%%%%%%%
\subsection{Motivation}
%%%%%%%%%%%%%%%%%%%%%%%%%%%%%%%%%%%%%%%%%%%%%%%%%%%%%%%%%%%%
\label{sec:motivation}

Non-backtracking walks seem to be an essential part of various aspects of
graph theory, including the Broder-Shamir trace method
(\cite{BS87,Fri91,Fri93,Fri}), Ihara zeta functions of graphs
(\cite{Ha89,KS00,ST}), Bollobas' question concerning girth and average
degree (\cite{AHL02,Hoo05}), and decoding of low density linear codes
(\cite{Gal63,Sh04,RU}). Given $H$, there is a matrix, $B=B(H)$, whose
powers count non-backtracking walks of a given length. $B$ is indexed by
the directed edges of $H$, and is analogous to the adjacency matrix for
paths (see below). Many interesting quantities concerning non-backtracking
walks can be expressed in terms of spectral properties of $B$.

Several applications of walks in a graph $H$, require understanding the
spectrum of the universal cover of $H$, denoted $\tH$. Similarly, there are
connections between non-backtracking walks on $H$ and the spectrum of the
non-backtracking walk operator on $\tH$.

For regular graphs, i.e., graphs in which each vertex has the same degree,
there is a direct and simple relationship between the spectrum of $B$ and
that of the adjacency matrix, $A=A(H)$. This fact renders many questions
quite easy for regular graphs, masking formidable difficulties that may
occur for non-regular graphs. In the case of a $d$-regular graph $H$, the
universal cover is a $d$-regular tree, and the spectrum of $\tA=A(\tH)$ is
known to be the interval
\[
\sigma(\tA) = \left[ -2\sqrt{d-1},2\sqrt{d-1} \right].
\]
In this case, the spectrum of $\tB=B(\tH)$ is the union of a circle and two
real intervals in the complex plane. Numerical experiments, such as those
of Stark and Terras (see \cite{ST} and the previous papers in that series),
suggest that $\tB$, for general (non-regular) $H$, can be two dimensional.
As an example, when $H$ is $K_4$ minus an edge, the non-backtracking
spectrum of $\tH$ is depicted in Figure~\ref{fig-tree-spec}.

The main goal of this paper is to give a method for finding the spectrum of
$\tB$ for any finite graph, $H$. This will prove the above noted
two-dimensionality. A consequence of this method is that $\rho(\tB) =
\sqrt{\gr}$, i.e., the spectral radius of $\tB$ is the square root of the
growth rate of $\tH$. This further motivates the definition of the graph
theoretical Riemann hypothesis proposed by Stark and Terras~\cite{ST}.

\subsubsection*{Alon's conjecture.}

Alon's conjecture states that for any $d\ge3$ and any $\eps>0$, 
the probability that a random
$d$-regular graph on $n$ vertices has second largest
eigenvalue $\lambda_2 \le 2\sqrt{d-1} + \eps$ tends to $1$ as
$n\to\infty$. This was proven for various
of notions of ``random graph,'' i.e., various probability distributions on
$d$-regular graphs on $n$ vertices, in \cite{Fri} (see also
\cite{BS87,Fri91}). A main tool used to this end is the analysis of
non-backtracking walks in random, regular graphs.

As we shall see (and see \cite{Fri}), for finite $d$-regular graphs each
$\mu\in\sigma(A)$ corresponds to two $\lambda\in\sigma(B)$ given by the
solutions of $\lam^2 - \mu\lam + d-1=0$. In particular, since $\sigma(A(H))
\subset [-d,d]$, the non-backtracking spectrum is contained in the union of
the circle $\{z : |z|=\sqrt{d-1} \}$ and the two real segments $[1-d,-1]$
and $[1,d-1]$. This simple relation yields a form of Alon's conjecture in
terms of the non-backtracking spectrum, namely that the spectral radius,
$\rho(B)$, of $B$ is at most $\sqrt{d-1}+\eps$.

In \cite{Fri03}, Alon's conjecture is relativized, meaning that for any
graph, $H$ (the ``base graph''), we consider a random lift, $\pi\from G\to
H$, of degree $n$; eigenpairs of $A(H)$ pull back to eigenpairs of $A(G)$,
we call such eigenpairs of $A(G)$ {\em old} eigenpairs of $\pi$, while
eigenpairs of $A(G)$ not arising from $A(H)$ we call {\em new} eigenpairs
of $\pi$; we conjecture that for any fixed $H$ and $\eps>0$, most random
lifts have all new eigenvalues of absolute value bounded above by $\eps$
plus the spectral radius of $A(\tH)$; and a weaker new eigenvalue bound is
proved via the Broder-Shamir technique
(Alon's original conjecture is the case where $H$
has one vertex). The full conjecture might be provable with a refined
Broder-Shamir method (as in \cite{Fri91,Fri}) that seems to require a
translation between the adjacency matrix spectrum and non-backtracking
spectrum. However, for irregular graphs we are aware of no such
translation. This makes it difficult to apply the trace method in the
general case.

Computer experiments indicate that the new (i.e., not arising from the
base) non-backtracking spectrum of large random lifts is distributed in
some fixed region, as seen in Figure~\ref{fig-lift-spec}. We would like to
prove that this region is the non-backtracking spectrum of the universal
cover of the base graph, shown in Figure~\ref{fig-tree-spec}. This suggests
a generalization of the relative Alon conjecture, roughly stating that the
new non-backtracking spectrum of a large random lift is near the
non-backtracking spectrum of the universal cover; for non-regular graphs,
unlike the case of regular graphs, it is not clear how this relates to the
usual relative Alon conjecture regarding adjacency matrix spectrum. This
might allow us to apply the trace methods of \cite{BS87,Fri93,Fri} in order
to prove some analog of Alon's conjecture for lifts of general base graphs
(perhaps first for the non-backtracking spectrum and then, ultimately, on
the adjacency matrix spectrum).

The relativized Alon conjecture raises the questions of how to determine
the non-backtracking spectrum of the universal cover of a finite graph
(answered in this paper) and how can one translate spectral bounds between
adjacency matrices and non-backtracking adjacency matrices in the irregular
case.

\subsubsection*{The Ihara Zeta function.}

For a finite graph $H$, the eigenvalues of $B(H)$ are the reciprocals of
the poles of the {\em Ihara zeta function} of the graph. This zeta function
is a graph theoretic analogue of the number theoretic Riemann zeta
function. For further details, we refer the reader to
Hashimoto~\cite{Ha89}, Kotani and Sunada~\cite{KS00}, and to Stark and
Terras~\cite{ST}.

\subsubsection*{Bounds for graphs with large girth.}

A third motivation is the result of Alon, Hoory, and Linial~\cite{AHL02},
giving a lower bound on the number of vertices in a graph with a specified
girth and average degree. They answer affirmatively a longstanding open
question by Bollob{\'a}s. Their result is basically a lower bound on the
Perron eigenvalue of $B$ in terms of the average degree. In further work,
Hoory~\cite{Hoo05}, gave a similar bound on the spectral radius of $H$'s
universal cover, $\rho(A(\tH))$.

Further questions raised by these works concern the relation between
$\tH$'s growth rate $\gr=\rho(B)$, and the spectral radii of the remaining
three operators $A$, $\tA$, and $\tB$. In this paper we partially answer
this question by proving that $\rho(\tB)=\sqrt{\gr}$.

\subsubsection*{LDPC error correcting codes.}

A final motivation is the analysis of a decoding algorithm for a family of
error correcting codes, known as low density parity check (LDPC) codes.
This decoding algorithm, which is receiving a lot of attention at present,
was first suggested by Gallager~\cite{Gal63}. LDPC codes are binary linear
codes of length $n$, described by a set of highly sparse linear equations
over $GF(2)$. It is convenient to represent such a code by a bipartite
graph, where the two sides correspond to variables and equations. Each
equation is adjacent to the variables occurring in it. There is a natural
decoding algorithm for such codes based on message passing. At each
iteration of the algorithm, a message is sent along each directed edge of
the graph (alternating between variables to equations and equations to
variables rounds). We refrain from further description of this algorithm,
and refer the reader to a survey by Shokrollahi~\cite{Sh04}, or a book by
Richardson and Urbanke~\cite{RU}.

The connection of this algorithm to non-backtracking walks is in the fact
that the $k$-th generation message computed at edge $e$ depends only on the
initial value at the endpoints of all length $k$ non-backtracking walks
starting at $e$. It seems that to analyze such an algorithm, one needs a
non-backtracking analog of expansion. For example, the Mixing Lemma for
expanders states that the number of edges $e(S,T)$ between any two vertex
sets $S$, $T$ in a $d$-regular graph satisfies $\Big|e(S,T)-d|S||T|/N\Big|
\le \lam_2 \sqrt{|S||T|}$, where $\lam_2$ is the second largest eigenvalue
in absolute value (see \cite{HLW} for a survey on expander graphs). It
would be helpful to develop a non-backtracking analogue of the expander
Mixing Lemma.

\subsection{Main results}

In the following unless otherwise specified, all graphs will be {\em
  simple}, meaning without self-loops or multiple edges, and connected; in
Section~9 we explain our conventions and generalizations to graphs that are
not simple. Throughout this paper, if $G$ is a graph, then $V(G)$ denotes
$G$'s vertices, and $E(G)$ denotes $G$'s {\em directed edges}, which (for
$G$ simple) we define as the union of $\{(u,v),(v,u)\}$ over all $u,v$
joined by an edge of $G$; for $e=(u,v)\in E(G)$ we set $e^{-1}=(v,u)$; we
write $u\sim v$ to denote that there is an edge joining $u$ and $v$. (All
graphs in this paper are locally finite, meaning that each vertex is
incident upon a finite number of edges.)

Let $\pi\from T\to H$ be the universal cover of a finite graph, $H$. For an
operator, $M$, on $l^2(V(T))$, we set $M_{uv} = \langle \delta_u, M
\delta_v \rangle$ for $u,v\in V(T)$ (where $\delta$ denotes the Dirac delta
function); we say that $M$ is
\begin{enumerate}
\item {\em symmetric} if $\overline{M}^*=M$;
\item {\em local} if $M_{uv}=0$ for $u,v$ of distance greater than one;
\item{\em connected} if $M_{uv}\ne 0$ for $u,v$ of distance one;
\item {\em a pullback (from $\pi$ or, abusively, $H$)} if $M_{uv}$ depends
  only on $\pi(u)$ and $\pi(v)$ assuming $u\sim v$.
\end{enumerate}
The main result of this paper gives a way to compute the spectrum of local,
symmetric operators on $T$ pulled back from $H$. If $M$ has an inverse, we
define the {\em Green function for $M$} to be $G(u,v)=\langle \delta_u,
M^{-1} \delta_v \rangle$. If an operator $A$ on $l^2(V(H))$ pulls back to
$\pi^*A$ on $T$, then determining the spectrum of $\pi^*A$ amounts to
determining for which $\lambda$ the operator $M=M_\lambda=\pi^*A-I\lambda$
has a Green function.

The key idea in this paper to determine the Green function, $G$, of $M$, is
that for any edge, $e=(u,v)$, of $E(T)$, the ratio $r_e=G(u,v)/G(u,u)$
should depend only on $\pi(e)$. This surprised us at first, although the
intuitive explanation is simple. If $\pi(e)=\pi(e')$ for $e'=(u',v')\in
E(T)$, then there is an automorphism of $T$ taking $e$ to $e'$; using this
automorphism we can ``graft'' appropriate linear combinations of pieces of
$G(u,\,\cdot\,)$ and $G(u',\,\cdot\,)$ onto $G(u',\,\cdot\,)$; if $r_e\ne
r_{e'}$, then this ``grafting'' process will yield more than one Green
function, which is impossible (see Lemma~\ref{L:one_dim}).

There are numerous technical difficulties in turning this intuition into
precise theorems. One is that an example of Fig\`a-Talamanca and Steger
(see Section~\ref{se:examples}) shows that $G(u,u)$ may vanish, forcing us
to work with the more delicate situation of zero and infinite ratios.
Another is that knowing $M$ and a ratio system, $r$, while it is easy to
determine $G$, it does not seem easy to know if $M$ is invertible, at least
when $H$ is infinite.

To make a precise statement, if $a,b\in\C$ with at least one of $a,b$ not
vanishing, we interpret $b/a$ to be $\infty$ if $a=0$ and the usual ratio
otherwise; in this case we say the {\em generalized ratio}, $b/a$, exits.
If $u,v,w\in V(T)$ are vertices of a tree, $T$, we write $[u,w]$ for the
unique path (a sequence alternating between vertices and edges) from $u$ to
$w$ in $T$, and say that $v$ is between $u$ and $w$ if $v$ lies in $[u,w]$.

\pagebreak

\begin{defn}\label{D:generalized-ratio-system}
  A {\em {\generalizedRS} for $M$} is a function $r\colon E(H) \to
  \C\cup\{\infty\}$ satisfying the following:
  \begin{enumerate} \renewcommand{\labelenumi}{(\alph{enumi})}
  \item If $r_e \ne 0$ for some $e=(u,v)$ then
    \begin{equation} \label{eq:recursion}
      0 = m_{vv} + \frac{m_{e^{-1}}}{r_e}
      + \sum_{e'\text{ s.t. } e\to e'} m_{e'} r_{e'},
    \end{equation}
    with the convention that $c/\infty=0$ for any $c$, and where $e\to e'$
    means that $e'\neq e^{-1}$ but $(e,e')$ forms a path of length two
    (i.e., $e$'s head is the tail of $e'$).
  \item If $r_e=0$ then there is some $f$ with $e\to f$ and $r_f=\infty$.
  \item If $r_f=\infty$ and $e\to f$, then $r_{e}=0$ and $r_{f^{-1}}\ne0$.
  \end{enumerate}
\end{defn}
Here condition~(a) comes from the fact that for fixed $u$ we have
$MG(u,\,\cdot\,)=\delta_u$; conditions~(b) and (c) are needed when the
ratios become zero or infinite, i.e., when $G(u,u)=0$ for some $u$. Note
that there is no assumed relationship between $r_e$ and $r_{e^{-1}}$.

For a {\generalizedRS} we define $r$'s {\em decay rate}, $\alpha(r)$, to be
the Perron eigenvalue of the matrix $R$ whose rows and columns are indexed
by the directed edges of $H$ with nonzero (but possibly infinite) ratios,
defined by:
\begin{equation}\label{eq:defineR}
  R_{e,f} = \begin{cases}
    \displaystyle
    \sum_{e' \text{ s.t. } e \to e' \to f}
                |m_f/m_{e'^{-1}}|^2 & r_{f}=\infty, \\
    |r_f|^2 & \text{$r_f \ne \infty$ and $e\to f$,} \\
    0 & \text{otherwise.}
  \end{cases}
\end{equation}

\begin{thm}\label{T:main}
  Let $\pi\from T\to H$ be a universal cover of a finite graph, $H$, and
  let $M$ be a pulledback, symmetric, local operator on $l^2(V(T))$. Then
  $M$ has a bounded inverse if and only if $M$ has a {\em ratio system} $r
  : E(H) \to \C \cup \{\infty\}$ with $\alpha(r)<1$.

  Furthermore, for and edge $(u,v)$ of type $e$ and any $w$ such that $u\in
  [w,v]$, we have $G(w,v)/G(w,u)=r_e$ whenever this generalized ratio is
  defined.
\end{thm}

Theorem~\ref{T:main} has numerous applications, including to adjacency
matrices and Laplacians. The application to non-backtracking spectrum of a
graph, $G$, derives from its connection to the local operator
\[
Q_\lambda = Q-\lambda A+\lambda^2 I
\]
where $A=A_G$ is the adjacency matrix of $G$ and $Q=Q_G$ is the diagonal
matrix with entries $Q_{vv}=d_v-1$ where $d_v$ is the degree of $v$.
Specifically we shall prove the following theorem for graphs, $G$, of
bounded degree, regarding the spectrum, $\sigma(B)$, of $B=B(G)$.

\begin{thm}\label{T:small_spec}
  Let $G$ be a graph with degrees bounded by $\dmax$, then
  \[
  \sigma(B(G)) = \{\pm 1\} \cup \{ \lam : Q_\lam \text{ is not invertible}
  \}.
  \]
\end{thm}

For $\lambda=0$, $Q_\lambda$ is symmetric, and otherwise
\[
\lambda^{-1}Q_\lambda=\lambda^{-1}Q - A + \lambda I
\] 
is (real) symmetric, and so Theorem~\ref{T:main} applies to $Q_\lambda$.

The following provides an interesting contrast to Theorem~\ref{T:main}.

\begin{thm}\label{T:alpha1general}
  Let $\pi\from T\to H$ be as in Theorem~\ref{T:main}. If $r$ is a
  {\generalizedRS} with $\alpha(r)=1$, then $M$ is not invertible.
\end{thm}
Calculations with $K_4$ minus an edge, as in Figure~\ref{fig-tree-spec},
show that a given $M$ can have more than one ratio system.
Theorem~\ref{T:alpha1general} says that if there is a ratio system with
$\alpha=1$, then there are none with $\alpha<1$.

The rest of this paper is divided as follows. In Section~2 we give some
preliminary remarks. In Section~3 we discuss the non-backtracking spectra
of trees and prove Theorem~\ref{T:small_spec}. In Section~4 we bound the
non-backtracking spectrum in an annulus in terms of the growth rate of the
tree. In Section~5 we discuss the Green function and relative kernels (that
make precise the ``grafting'' idea discussed earlier showing that Green's
function has ``periodic ratios''). In Section~6 we prove
Theorem~\ref{T:main} in the case of a ratio system with no zero or
infinite ratios. In Section~7 we prove Theorem~\ref{T:main} in full
generality. In Section~8 we give some interesting examples illustrating
various aspects of Theorem~\ref{T:main}. In Section~9 we show how to
extend our theory to much more general base graphs (or pregraphs), $H$. In
Section~10 we make some remarks, mentioning a few of the many questions
left open by this paper.

\subsection*{Acknowledgments}

We wish to thank Chen Greif, Audry Terras, Alex Gamburd, and B\'alint
Vir\'ag for illuminating discussions.

%%%%%%%%%%%%%%%%%%%%%%%%%%%%%%%%%%%%%%%%%%%%%%%%%%%%%%%%%%%%
\section{Preliminaries}
%%%%%%%%%%%%%%%%%%%%%%%%%%%%%%%%%%%%%%%%%%%%%%%%%%%%%%%%%%%%
\label{sec:prelim}

A directed edge $e'=(u',v')$ of a graph, $H$, {\em can follow} edge
$e=(u,v)$, denoted $e\to e'$, if $u'=v$ and $e'\ne e^{-1}$. A
non-backtracking walk on $H$ is a sequence of directed edges
$e_1,e_2,\dots,e_r$ such that $e_k\to e_{k+1}$ for all $k$.

If $H$ is a graph and $v\in V(H)$, the {\em 1-neighbourhood of $v$} is the
subgraph consisting of all undirected edges incident upon $v$. A graph
homomorphism $\nu\colon H'\to H$ is a {\em covering map} (see \cite{Fri93})
if for each $v'\in V(H')$, $\nu$ gives a bijection of the edges of the
1-neighbourhood of $v'$ with those of $v$ (note that this definition is
correct even when the graphs have self-loops; see Section~\ref{se:self}).
$H'$ is a {\em cover} of $H$. The {\em type} of a vertex $v\in H'$, is
defined as the vertex $\nu(v)\in H$. If $\nu(v')=v$ for vertices $v',v$, we
say that $v'$ is {\em above} $v$ (or $v$ is {\em below} $v'$). These terms
are similarly defined for edges.

Note that a cover of a cover is a cover. The {\em universal cover} of $H$,
denoted $\tH$, is the (unique up to isomorphism) cover of $H$ that is also
a cover of every other cover of $H$. The universal cover may be constructed
as follows. Fix some (arbitrary) vertex $v_0 \in H$. The vertices of $\tH$
are the finite non-backtracking walks on $H$ starting at $v_0$. Two
vertices of $\tH$ are adjacent if one extends the other by a single step.
The covering map $\pi\from V(\tH)\to V(H)$ maps a walk to its terminal
vertex in $H$. It is easy to see that $\tH$ is a tree, hence it is also
referred to as the covering tree of $H$. The universality of $\tH$ can be
stated as follows: If $\nu\colon H'\to H$ is any covering map, and $v'\in
V(H')$ is of type $v_0$, then there is a unique covering map
$\mu\colon\tH\to H'$ with $\pi=\nu\mu$ and $\mu(\phi)=v'$, where $\phi$ is
the empty path.)

If the covering map $\nu\colon H'\to H$ is $n$-to-1 for some finite $n$,
then we also say that $H'$ is a {\em lifting}, or an {\em $n$-lifting}, of
$H$. Every cover is either a lifting or is $\infty$-to-1. Equivalently, the
graph $H'$, which we denote $H^n$, can be characterized as follows: the
vertex set of $H^n$ consists of $n$ copies of $H$'s vertices: $V(H^n) =
V(H) \times [n]$, with the covering map $\pi((v,i))=v$. The edges in $H^n$
are given by a perfect matching between $\pi^{-1}(v)$ and $\pi^{-1}(u)$ for
every undirected edge $(u,v) \in E(H)$. This characterization gives rise to
a natural distribution on $n$-lifts of $H$, by picking the perfect
matchings independently and uniformly at random.

We define 0-forms to be functions on the vertices of a graph, and 1-forms
to be functions on directed edges of a graph (with no constraint relating
$f(e)$ and $f(e^{-1})$.) Denote by $A=A(H)$ the adjacency operator acting
on the Hilbert space of $0$-forms, i.e., $(Af)(v) = \sum_{u\sim v} f(u)$
with multiple edges taken into account.) For graphs with minimal degree at
least two, we consider the non-backtracking adjacency operator, $B=B(H)$,
which acts on 1-forms. $B$'s action is given by
\[
(Bf)(e) = \sum_{e' \text{ s.t. } e\to e'} f(e').
\]
The definitions of $A,B$ are valid for any locally finite graph, and in
particular for $\tH$. When the degrees in a graph are globally bounded, $A$
and $B$ may be viewed as bounded operators on $l^2(V)$ and $l^2(E)$,
respectively. In the finite case, $B$ is the adjacency matrix of the
directed graph with an edge from $e_1$ to $e_2$ iff $e_2\to e_1$. Notice
that in the literature $B$ is occasionally defined as the adjoint of our
definition.

The {\em non-backtracking spectrum} of a graph $H$ is defined to be the
spectrum of the corresponding operator $B$ (as opposed to the spectrum of
$H$ which is the spectrum of $A$.) Note that $A$ is self-adjoint, hence any
graph's spectrum is real. However, the non-backtracking spectrum generally
contains complex numbers. $\tA=A(\tH)$ and $\tB=B(\tH)$ denote the
corresponding operators on $\tH$.

We denote by $\rho(u,v)$ the graph distance in $\tH$. We let $\gr$ denote
the growth rate of the universal cover $\tH$, i.e.
\[
 \gr = \lim_{r\to\infty} |{\tt Ball}(v,r)|^{1/r},
\]
where ${\tt Ball}(v,r)$ is the radius $r$ ball around some (arbitrary)
vertex $v \in V(\tH)$; for connected $H$ the limit is independent of the
choice of $v$. It is straightforward to see $\gr$ is the Perron-Frobenius
eigenvalue of $B$.

%%%%%%%%%%%%%%%%%%%%%%%%%%%%%%%%%%%%%%%%%%%%%%%%%%%%%%%%%%%%
\section{Non-backtracking spectrum basics}
%%%%%%%%%%%%%%%%%%%%%%%%%%%%%%%%%%%%%%%%%%%%%%%%%%%%%%%%%%%%
\label{se:spectrum_basics}

We start with some background from spectral theory. Let $T$ be a bounded
linear operator acting on a Hilbert space $\mathcal{H}$ (in this paper we
work with $l^2(X)$, where $X$ will be one of $V$, $\tV=V(\tH)$, $E$, or
$\tE=E(\tH)$). Recall that the spectrum of $T$ is
\[
\sigma(T) = \{\lam | \text{$T - \lam I$ does not have a bounded inverse}\},
\]
and its spectral radius is
\[
\rho(T) = \max\{|\lambda| : \lambda\in\sigma(T)\}
\]
(no confusion should arise regarding this $\rho$ and the distance, $\rho$,
that involves two arguments). If $\mathcal{H}$ is finite dimensional, then
$\lam \in \sigma(T)$ iff there is a nonzero $f$ satisfying $Tf=\lam f$. In
the general case we have \cite[Chapter~12]{Rud91}:

\begin{thm}\label{T:operator-spectrum}
  For a bounded linear operator $T\colon \mathcal{H}\to\mathcal{H}$, we
  have $\lam\in\sigma(T)$ iff one of the following three possibilities
  hold:
  \begin{enumerate} \romenum
  \item there exists a nonzero $f$ such that $(T-\lam I)f=0$,
  \item there exists a nonzero $f$ such that $(T^*-\overline{\lam} I)f=0$,
  \item for every $\eps>0$ exists a nonzero $f$ such that $\|(T-\lam I)f\|
    < \eps \|f\|$.
  \end{enumerate}
\end{thm}

In case (iii), a sequence of $f$'s with $\eps\to0$ is referred to as {\em
  approximate eigenfunctions}; case~(i) is a special case of (iii).

\begin{defn}
  A bounded linear operator $T$ is called {\em symmetric} (sometimes
{\em real symmetric} in the literature) if
  $\overline{T}^*=T$.
\end{defn}

If $T$ is symmetric, then case~(ii) of \thmref{operator-spectrum} is
equivalent to case~(i), and so $\lam$ is in the spectrum iff case~(iii)
holds, i.e., we have a sequence of approximate eigenfunctions for $\lam$.

While the operator $B$ is not symmetric, to identify its spectrum it
suffices to check only case~(iii). To see this, define the unitary operator
$U\colon l^2(E)\to l^2(E)$ by its action $(Uf)_e = f_{e^{-1}}$. Note that
$U^{-1}=U$. Since $B$ is real, it is easy to see that $U^{-1}BU = B^T =
B^*$. Consequently, if $f$ is an eigenfunction of $B^*$ with eigenvalue
$\overline{\lam}$, then $Uf$ is an eigenfunction of $B$ with eigenvalue
$\lam$. Thus case~(ii) of \thmref{operator-spectrum} holds for the same
$\lam$ as case~(i), and it suffices to check case~(iii).

\medskip

At this point we are ready to prove Theorem~\ref{T:small_spec}.  Before
doing so we make a few remarks on the theorem.

Note that if $G$ is a $d$-regular graph, $Q=(d-1)I$ and so
$Q_\lam=(\lam^2+d-1)I-\lam A$ is singular iff $\lam^2-\mu\lam+d-1=0$ for
some $\mu\in\sigma(A)$. \thmref{small_spec} is known (in more precision)
for finite graphs:

\begin{thm}[Bass~\cite{Ba92}, Kotani and Sunada~\cite{KS00}]
  \label{T:Q_finite}
  For any finite graph
  \[
  \det(B-\lam I) = (\lam^2-1)^{|E|-|V|} \cdot \det(Q_\lam).
  \]
\end{thm}

\begin{note}
  A version of \thmref{Q_finite} holds for graphs with multiple edges and
  self-loops. Each half-loop (see Section~\ref{se:self}) contributes a
  factor of $\lam-1$, rather than $\lam^2-1$.
\end{note}

We thank Chen Greif for the following remark:
\begin{note}
  Consider the $2n\times2n$ matrix $X=\left(\begin{smallmatrix} A&-Q\\I&0
    \end{smallmatrix}\right)$. One can easily verify by row elimination
  that $\det(Q_\lam)=\det(X-\lam I)$. This observation has practical
  importance for calculating the non-backtracking spectrum of specific
  graphs, since eigenvalue calculation is a standard feature in most
  numeric/symbolic matrix manipulation programs.
\end{note}

\begin{proof}[Proof of \thmref{small_spec}]
  If $G$ is finite, then \thmref{Q_finite} implies our result, hence we may
  restrict ourselves to infinite $G$. The result follows from the following
  three statements:
  \begin{enumerate}
  \item we have $\pm1 \in \sigma(B)$;
  \item if $\lam\ne\pm1$ and $Q_\lam$ is not invertible, then neither is
    $B-\lam I$;
  \item if $\lam\ne\pm1$ and $B-\lam I$ is not invertible, then neither is
    $Q_\lam$.
  \end{enumerate}

  To prove (1), we explicitly construct approximate eigenfunctions. $G$
  contains some infinite path of directed edges $\{e_i\}_{i\ge0}$. For any
  $k$, define the function
  \[
  f = \sum_{i=1}^{k} \delta_{e_i} - \sum_{i=0}^{k-1} \delta_{e_i^{-1}},
  \]
  i.e., $f$ is 1 on $e_1,\ldots,e_k$, $f$ is $-1$ on $e_0^{-1}, \dots,
  e_{k-1}^{-1}$, and zero on all other edges. One can easily check that
  \[
  (B-I)f = \delta_{e_0} - \delta_{e_0^{-1}} - \delta_{e_k} +
  \delta_{e_k^{-1}}.
  \]
  Thus $\|(B-I)f\| = \sqrt{4}$, while $\|f\|=\sqrt{2k}$, and so taking
$k\to\infty$ gives
approximate eigenfunctions for
  $\lam=1$.

  $\lam=-1$ is similar, with
  \[
  f = \sum_{i=1}^{k} (-1)^i \delta_{e_i} + \sum_{i=0}^{k-1} (-1)^i
  \delta_{e_i^{-1}}.
  \]

  To prove (2), assume $Q_\lam$ is not invertible. Since $Q_\lam$ is
  symmetric, as noted above
  it has approximate (perhaps exact) eigenfunctions. Thus for every $\eps$
  there is an $f$ such that $\|f\|=1$ and $\|Q_\lam f\|<\eps$. Note that
  $f$ is defined on the vertices of $G$. From $f$ we construct a 1-form,
  $g$, given by $g_{uv} = f_u - \lam f_v$. Then:
  \begin{align*}
    \big((B-\lam I)g\big)_{uv}
    & = \sum_{\substack{w\sim v\\w\ne u}} g_{vw} - \lam g_{uv}
      = \sum_{w \sim v} g_{vw} - g_{vu} - \lam g_{uv}  \\
    & = \sum_{w \sim v} (f_v - \lam f_w) + (\lam^2-1)f_v  \\
    & = (d_v-1)f_v - \lam (Af)_v + \lam^2 f_v = (Q_\lam f)_v.
  \end{align*}
  Therefore,
  \[
  \|(B-\lam I)g\|^2 \le \dmax \cdot \|Q_\lam f\|^2 < \eps^2 \dmax.
  \]
  It remains to show that $\|g\|$ is bounded from below. To this end, note
  that
  $\left(\begin{smallmatrix} g_{uv} \\ g_{vu} \end{smallmatrix}\right) =
  \left(\begin{smallmatrix} 1&-\lam\\ -\lam&1 \end{smallmatrix}\right)
  \left(\begin{smallmatrix} f_u \\ f_v \end{smallmatrix}\right)$
  is an invertible linear transformation for $\lam\ne\pm1$. Thus for
  $\lam\ne\pm1$, for some $c=c(\lam)>0$
  \[
  |g_{uv}|^2 + |g_{vu}|^2 \ge c\left(|f_u|^2+|f_v|^2 \right).
  \]
  Summing over all edges we find $\|g\|^2 \ge c$.

  To prove (3), let $S_o$, $S_i$ be the outbound and inbound sum operators
  from 1-forms to 0-forms, defined by $(S_o h)_v = \sum_{u \sim v} h_{vu}$,
  and $(S_i h)_v = \sum_{u \sim v} h_{uv}$. For any $v$, and any 1-form
  $h$:
  \begin{align*}
    (S_i Bh)_v
    = \sum_{u \sim v} (Bh)_{uv}
    &= \sum_{\substack{u,w \sim v \\w\ne u}} h_{vw}
    = (d_v-1) (S_o h)_v \\
    (S_o Bh)_v
    = \sum_{u \sim v} (Bh)_{vu}
    &= \sum_{\substack{w\sim u\sim v\\w \ne v}} h_{uw}
    = (A S_o h - S_i h)_v.
  \end{align*}
  And so we have
  \[
    S_i B = Q S_o,  \qquad \text{ and } \qquad   S_o B = A S_o - S_i.
  \]

  Assume $B-\lam I$ is not invertible for some $\lam\ne\pm1$. As before,
  since $B$ is real, for any $\eps>0$ there is a function $g$ such that
  $\|g\|=1$, and $\|(B-\lam I)g\|<\eps$. We claim that
  $f = S_o g$ is an approximate eigenfunction of $Q_\lam$. Indeed:
  \begin{align*}
    \|Q_\lam f\|
    & = \|(\lam^2 I - \lam A + Q) S_o g\|  \\
    & = \|(\lam^2 S_o - \lam S_o B - \lam S_i + S_i B)g\|  \\
    & = \|(S_i-\lam S_o)(B-\lam I)g\|  \\
    & \le (1+|\lam|) \sqrt{\dmax} \, \eps
  \end{align*}
  (since $\|S_o\|=\|S_i\| = \sqrt\dmax$). Furthermore, for $\lam\ne\pm1$ we
  shall bound $\|f\|$ from below. For any edge $vu$,
  \[
  (S_o g)_v = (Bg)_{uv} + g_{vu} -\lam g_{uv} + \lam g_{uv};
  \]
  notice that $S_i^*$ is the ``head map,'' $(S_i^*h)_{uv}=h_v$ for all
  $h\in L^2(V)$, and similarly $S_o^*$ is the ``tail map.'' So we may write
  the last equation as $S_i^* S_o = (B-\lam I)+(\lam I+U)$ (recall
  $(Ug)_{uv}=g_{vu}$.) Since $\|g\|=1$, $U^2=1$ and $\|U\|=1$, we have
  \begin{align*}
  |\lam^2-1| = \|(\lam^2-1)g\| &= \|(\lam-U)(\lam+U)g\| \\
  &\le (|\lam|+1)\|(\lam+U)g\|.
  \end{align*}
  Since clearly $\|h\|^2\ge \|S_i^* h\|^2/\dmax$ for any $h\in L^2(V)$, we
  have
  \begin{align*}
    \sqrt{\dmax} \|f\| \ge \|S_i^*f\| = \|S_i^* S_o g\|
    &\ge \|(\lam I+U)g\|-\|(B-\lam I)g\| \\
    &\ge \frac{|\lam^2-1|}{|\lam|+1} - \|(B-\lam I)g\|.
  \end{align*}
  Since $\|(B-\lam I)g\|<\eps$, it follows that for $\eps$ small enough
  $\|f\|$ is bounded from below.
\end{proof}

\begin{note}
  It is possible to prove \thmref{Q_finite} by a similar argument, with
  cycles taking the place of the infinite path, resulting in exact
  eigenfunctions; the eigenspace of 1 is spanned by functions which are 1
  on edges of an oriented cycle and -1 on the reverse cycle. The eigenspace
  of -1 is similar but with alternating signs along cycles of even length.
\end{note}

The following theorem restricts the non-backtracking spectrum of $G$ to the
union of an annulus in the complex plain, and two real intervals. For
finite graphs, the theorem is due to Kotani and Sunada~\cite{KS00}.

\begin{thm} \label{T:annulus}
  Let $\dmin \ge 2$ and $\dmax$ be the minimal and maximal degrees of some
  (finite or infinite) graph $G$. Then
  \[
  \sigma(B) \subset \left\{\lam\in\C :
      \sqrt{\dmin-1} \le |\lam| \le \sqrt{\dmax-1}\right\} 
    \bigcup \left\{\vphantom{\sqrt{\dmin}}
      \lam\in\R : 1 \le |\lam| \le \dmax-1 \right\}.
  \]
\end{thm}

\begin{proof}
  We first prove that
  \[
  \sigma(B) \subset \left\{\lam\in\C : 1 \le |\lam| \le \dmax-1\right\}.
  \]
  Fix some vertex $v$. Clearly $UB$ acts on functions supported on edges
  directed out of $v$ independently of its action on all other edges. On
  these $d_v$ edges, the action of $UB$ is given by the matrix $J-I$ where
  $J$ is the all ones matrix. Since $\sigma(J-I) = \{-1,d_v-1\}$ and since
  $J-I$ (being symmetric) is orthonormally diagonalizable, we easily see
  that
  \[
  \frac {\|(J-I)f\|} {\|f\|} \in [1,d_v-1]
  \]
  for any function $f$ supported on these edges;
  therefore, since $U$ is unitary, for any $1$-form, $f$, we have
  \[
  \frac{\|Bf\|}{\|f\|} = \frac{\|UBf\|}{\|f\|} \in [1,\dmax-1].
  \]
  If $\lam\in\sigma(B)$, then we have for any $\eps$ an approximate
  eigenfunction, $f$, with $\|f\|=1$ and $\|(B-\lam)f\|\le \eps$. So
  \[
  |\lam|=\|\lam f\| \le \| \lam f-Bf\| + \|Bf\| \le \eps+\dmax-1.
  \]
  Taking $\eps\to0$ yields $|\lam|\le\dmax-1$. Similarly
  \[
  |\lam| = \|\lam f\| \ge \| Bf\| - \| Bf-\lam f\| \ge 1-\eps,
  \]
  and we conclude $\|\lam\|\ge 1$.

  It remains to prove that the non-real $\lam \in \sigma(B)$ satisfy
  $\sqrt{\dmin-1}\le |\lam| \le \sqrt{\dmax-1}$. Given $f\in L^2(V)$,
  define
  \[
  \alpha = \frac{\langle f,Af\rangle}{\|f\|^2},   \quad
  \beta = \frac{\langle f,Qf\rangle}{\|f\|^2},
  \]
  and note that $\langle f,Q_\lam f\rangle = \|f\|^2\cdot(\lam^2 -
  \alpha\lam +\beta)$. Further, note that since $A,Q$ are real symmetric
  operators, that $\alpha,\beta$ are real. From the definition of $A,Q$, we
  clearly have $|\alpha|\le\dmax$ and $\beta \in [\dmin-1,\dmax-1]$.

  If $f$ is an eigenfunction then $\lam^2 - \alpha\lam +\beta=0$. If
  $Q_\lam$ has only approximate eigenfunctions, $f_\eps$, then for those we
  have $\alpha=\alpha_\eps,\beta=\beta_\eps$ satisfying
  \[
  |\lam^2 - \alpha_\eps \lam+\beta_\eps| = |\langle f,Q_\lam f \rangle |
  \le \eps
  \]
  and as before, $|\alpha_\eps|\le \dmax$ and $\beta_\eps \in
  [\dmin-1,\dmax+1]$. By passing to a subsequence we may assume that
  $\alpha_\eps$ and $\beta_\eps$ have limits, $\alpha,\beta$, as
  $\eps\to0$. We find that, again, $\lam^2-\alpha\lam+\beta=0$ with real
  $\alpha,\beta$ satisfying $\beta\in[\dmin-1,\dmax-1]$.

  Solving the quadratic equation we get $\lam=\alpha/2 \pm
  \sqrt{\alpha^2/4-\beta}$. If $\lam$ is not real then the square root is
  of a negative number and
  \[
  |\lam|^2 = (\alpha/2)^2 + \beta-\alpha^2/4 = \beta,
  \]
  and the claim follows.
\end{proof}

%%%%%%%%%%%%%%%%%%%%%%%%%%%%%%%%%%%%%%%%%%%%%%%%%%%%%%%%%%%%
\section{Non-backtracking spectrum of trees}
%%%%%%%%%%%%%%%%%%%%%%%%%%%%%%%%%%%%%%%%%%%%%%%%%%%%%%%%%%%%

The non-backtracking spectrum of a tree has a fourfold symmetry that
is not hard to rigorously prove:

\begin{prop}\label{P:four_symmetry}
  For any bipartite graph, $G$, $\sigma(B(G))$ is invariant under
  reflection in the real and imaginary axis.
\end{prop}

As we will be primarily concerned with the non-backtracking spectrum of
trees, note that this result applies whenever $G$ is a tree. 

\begin{proof}
  Since (for any graph) $B$ is a real operator, $\sigma(B)$ is invariant
  under complex conjugation.

  For a bipartite graph $G$, let $X,Y$ be the two sides of the graph.
  Consider the unitary operator $N$ on 1-forms that negates $f_{uv}$ for
  $u\in X$ and leaves $f_{uv}$ unchanged for $u\in Y$. Clearly $N^{-1}=N$
  and it is easy to see that $N^{-1}BN=-B$. Thus $\sigma(B) = -\sigma(B)$.
\end{proof}

\pagebreak

The next result improves on the upper bound of \thmref{annulus} for
sufficiently nice trees.  A graph has {\em uniformly bounded growth},
$\overline\gr$, if
\[
\overline\gr = \limsup M_r^{1/r}
\]
is finite, where
\[
M_r = \sup_v |{\tt Ball}(v,r)|.
\]

\begin{thm}\label{T:tree_annulus}
  Assume a tree without leaves has uniformly bounded growth $\overline\gr$,
  then
  \[
  \sigma(B) \subset \{ \lam : 1 \le |\lam| \le \sqrt{\overline\gr} \}.
  \]
\end{thm}

Note that for graphs with uniformly bounded degrees, $M_r =
O((\dmax-1)^r)$, and so $\overline\gr$ is finite. For universal covers of
finite graphs, $\overline\gr = \gr = \lim_{r\to\infty} |{\tt
  Ball}(v,r)|^{1/r}$ for any fixed $v$. We shall later see that for
universal covers of finite graphs, \thmref{tree_annulus} is sharp and in
fact $\{ \lam : |\lam|=\sqrt{\overline\gr} \} \subset \sigma(B)$.

\begin{proof}
  Denote $e \to_k e'$ if there is a non-backtracking path $e \to e_1 \to
  \cdots \to e_{k-1} \to e'$. Note that if such a path exists, then it is
  unique and its length is determined by $e,e'$. By definition,
  \[
  \|B^k\| = \sup_{\|f\|=\|g\|=1} \big| \langle g,B^kf \rangle \big|.
  \]
  For $f,g$ with $\|f\|=\|g\|=1$, and any collection of coefficients
  $a(e,e')>0$, we have by Cauchy-Schwarz 
  \begin{align*}
  \big| \langle g,B^k f \rangle \big|
     & = \left| \sum_{e'\to_k e} f(e) g(e') \right|           \\
     &\le \sum_{e'\to_k e}  a(e,e') |f(e)|^2 + a(e,e')^{-1} |g(e')|^2  \\
     &\le \left(\sup_e \sum_{e'\to_k e} a(e,e')\right) \|f\|^2
        + \left(\sup_{e'} \sum_{e'\to_k e} a(e,e')^{-1}\right) \|g\|^2.
  \end{align*}

  It remains to choose the coefficients $a(e,e')$ so that the sums above
  are small. To this end, fix some (arbitrary) vertex in the tree to be the
  root. The unique path from $e'$ to $e$ must descend toward the root some
  number of steps $i$, and then ascend $k-i$ steps on a different branch to
  reach $e$. Fix $a(e,e')=\overline\gr^{(i-k/2)}$.

  For any $\eps>0$ there is some $c=c(\eps)$, such that the total number of
  non-backtracking paths of length $\ell$ from any vertex is bounded by
  $c\cdot((1+\eps)\overline\gr)^\ell$. Since there is always a unique path
  toward the root, the number of paths starting at $e'$ with $i,k$ as above
  is at most $c\cdot((1+\eps)\overline\gr)^{k-i}$. Consequently
  \begin{align*}
    \sup_e \sum_{e'\to_k e} a(e,e')
    &\le \sum_{i=0}^k c((1+\eps)\overline\gr)^{k-i}
                                \overline\gr^{(i-k/2)} \\
    &\le c(k+1)(1+\eps)^k {\overline\gr}^{k/2}.
  \end{align*}
  Similarly, the number of paths of type $i$ ending at $e$ is bounded by
  $c_\delta((1+\eps)\overline\gr)^i$, and hence also
  \[
  \sup_{e'} \sum_{e'\to_k e} a(e,e')^{-1} \le c(k+1)(1+\eps)^k
  {\overline\gr}^{k/2}.
  \]

  Combining the bounds we find
  \[
  \rho(B) = \lim \big(\|B^k\|\big)^{1/k} \le (1+\eps) \sqrt{\overline\gr}.
  \]
  Since $\eps$ was arbitrary, this proves the claim.
\end{proof}

%%%%%%%%%%%%%%%%%%%%%%%%%%%%%%%%%%%%%%%%%%%%%%%%%%%%%%%%%%%%
\section{The Green function and the Relative Kernels}
%%%%%%%%%%%%%%%%%%%%%%%%%%%%%%%%%%%%%%%%%%%%%%%%%%%%%%%%%%%%

Let $\pi\from T\to H$ and $M$ as in Theorem~\ref{T:main}. By deleting
edges $e=\{u,v\}$ from $E(H)$ for which $m_{uv}=0$, we may assume $M$ is
connected, i.e., $m_{uv}\ne 0$ whenever $(u,v)\in E(V)$. If $e=(u,v)$ is an
edge of $T$ or $H$, we write $m_e$ for $m_{uv}$.

Assume that $M$ is invertible. Then there is a unique {\em Green function},
$G=G_M$, defined by the requirement that $MG(x,\cdot) = \delta_x(\cdot)$,
or equivalently,
\[
G(u,v) = (M^{-1}\delta_u)(v).
\]
Notice that the symmetry $M^*=\overline{M}$ implies the symmetry of the
Green function,
\begin{align*}
  G(u,v)= \langle M^{-1}\delta_u, \delta_v \rangle
  &= \langle \delta_u, (M^*)^{-1} \delta_v \rangle \\
  &= \overline{\langle (\overline{M})^{-1} \delta_v,\delta_u \rangle }
  = \langle M^{-1} \delta_v,\delta_u \rangle =G(v,u).
\end{align*}

\begin{prop}\label{pr:symmetry}
  Let $\pi\from T\to H$, $M$, and $G$, be as above. Let $\sigma$ be an
  automorphism of $T$ for which $\sigma^* M = M\sigma^*$, where $\sigma^*$
  is the pullback (i.e., $\sigma^* f = f\circ\sigma$). Then
  $G(x,y)=G(\sigma x,\sigma y)$ for all $x,y\in V(T)$.
\end{prop}

\begin{proof}
  We have $\delta_{\sigma z}=\sigma^*\delta_z$ for any $z\in V(T)$, and
  so
  \begin{align*}
    G(\sigma x,\sigma y) &= \langle M^{-1}\delta_{\sigma x},
                                           \delta_{\sigma y}\rangle
     = \langle M^{-1}\sigma^*\delta_x,\sigma^*\delta_y\rangle \\
    &= \langle (\sigma^*)^{-1} M^{-1}\sigma^*\delta_x, \delta_y\rangle
     = \langle M^{-1}\delta_x,\delta_y\rangle = G(x,y).
   \qedhere
   \end{align*}
\end{proof}

We introduce some further notions in order to derive additional properties
of the Green function. For a directed edge $e=(u,v)$ of the covering tree,
consider the subtree $T_e$ consisting of $v$'s connected component in
$T-e$, together with $e$ itself. Thus $T_e$ contains exactly those vertices
contained in a non-backtracking path from $u$ that begins with $e$. $T_e$
is considered to be rooted at $u$, which is of degree 1 in it. For any
vertex $x$, the unique neighbour closer to the root referred to as $x$'s
parent; the other neighbours are $x$'s children.

\begin{defn}
  Let $T_e$ be a tree as above, $e=(u,v)$. The {\em relative kernel of $M$
    with respect to $e$}, denoted $\relker{e}$, is the subspace of
  $l^2(V(T))$ consisting of functions $f$ supported on $T_e$ for which $Mf$
  is zero on $T_e\setminus\{u\}$. The {\em relative null-kernel}, denoted
  $\relnullker{e}$, is the subspace of the relative kernel consisting of
  those functions, $f$, for which $f(u)=0$.
\end{defn}

It is clear that either $\relker{e}=\relnullker{e}$ or else
$\relker{e}/\relnullker{e}$ is one-dimensional.

\begin{lemma}\label{L:one_dim}
  Let $M$ be invertible and let $e=(u,v)$.
  \begin{enumerate} \romenum
  \item If $G(u,u)\ne 0$, then $\dim(\relker{e}) = 1$.
  \item If $G(u,u)=G(u,v)=0$, then $\dim(\relker{e}) = 1$.
  \item If $G(u,u)=0$ and $G(u,v)\ne 0$, then $\dim(\relnullker{e}) = 1.$
  \end{enumerate}
\end{lemma}

This lemma shows that $\relker{e}$ is one-dimensional or two-dimensional,
with the latter only possible in the case where $G(u,u)=0$ and $G(u,v)\ne
0$. We do not know of any operator where $\relker{e}$ is two-dimensional,
or whether such an example exists. There is a class of $M$'s where there
are edges for which $G(u,u)=0$ and $G(u,v)\ne 0$ (see
Section~\ref{se:steger}).

\begin{proof}
  Assume first $G(u,u)\ne 0$. Consider the function
  \[
  f(w) = \begin{cases} G(u,w) & w\in T_e \\ 0 & w\notin T_e\end{cases}.
  \]
  $Mf$ differs from $MG(x,\cdot)=\delta_u$ only at $u$ and outside $T_e$,
  and therefore $f\in\relker{e}$. Since $f(u)\ne 0$, the dimension of
  $\relker{e}$ is at least 1. To show that the dimension is not larger than
  one, suppose that $g\in\relker{e}$, and define $\tilde g$ via
  \[
  \tilde g(w)= \begin{cases} G(u,u)g(w) & \text{if $w\in V(T_e)$,} \\
                             G(u,w)g(u) & \text{otherwise.}
  \end{cases}
  \]
  Thus $\tilde g$ is proportional to $g$ on $T_e$ and to $G(u,\cdot)$
  outside. The two definitions coincide at $u$. It follows that $M \tilde
  g$ is supported on $u$, and so $\tilde g$ is proportional to $G(u,\cdot)$
  and hence $g$ is proportional to $f$.

  \medskip

  Assume from here on that $G(u,u)=0$. Consider the restriction of the
  $G(u,\cdot)$ to the subtrees emerging from $u$. Since the Green function
  is unique, $G(u,\cdot)$ must vanish on all but one of these trees
  (otherwise the restriction of $G(u,\cdot)$ to each subtree gives another
  function with $Mf=\delta_u$). In particular, there is a unique neighbour
  $x$ of $u$ with $G(x,u)\ne0$.

  If $G(u,v)=0$, then the we repeat the argument above with $f$ the
  restriction to $T_e$ of $G(x,\cdot)$. As above, $f\in \ker_e$. If there
  is another $g\in\ker_e$, then let
  \[
  \tilde g(w)= \begin{cases} G(x,u)g(w) & \text{if $w\in V(T_e)$,} \\
                             G(x,w)g(u) & \text{otherwise.}
  \end{cases}
  \]
  This time, $M \tilde g$ is supported on $\{x,u\}$, and so $\tilde g =
  c_1G(x,\cdot) + c_2 G(u,\cdot)$. Applying this at $x$ and $u$ gives
  \begin{align*}
    G(x,x)g(u) &= c_1G(x,x) + c_2 G(u,x),  &
    G(x,u)g(u) &= c_1G(x,u) + c_2 G(u,u)
  \end{align*}
  and solving for $c_1,c_2$ shows $c_2=0$ (and $c_1=g(u)$). Thus $g$ is
  proportional to $f$, and $\dim(\ker_e)=1$.

  Finally, if $x=v$, then $G(u,\cdot)$ vanishes outside $T_e$, and so
  $f=G(u,\cdot) \in \relnullker{e}$ implying the dimension is at least one.
  Conversely, for any $g\in\relnullker{e}$ the function $M g$ is supported
  on $u$ and hence must be proportional to $f$, implying the dimension is
  exactly one.
\end{proof}

\begin{coro}\label{C:proportional}
  If $e=(u,v)$ and $x,y$ are not in $V(T_e)\setminus \{u\}$, then the
  restrictions of $G(x,\,\cdot\,)$ and $G(y,\,\cdot\,)$ to $T_e$ are
  proportional.
\end{coro}

\begin{proof}
  The cases $G(u,u)\ne 0$ and $G(u,u)=G(u,v)=0$ follow from the
  one-dimensionality of $\relker{e}$. The remaining case $G(u,u)=0$ and
  $G(u,v)\ne 0$ follows from the fact mentioned in the above proof that
  $G(u,\,\cdot\,)$ is supported on $T_e$, so $G(u,x)=G(u,y)=0$. Hence
  $G(x,u)=G(u,x)=0$. Thus the restrictions to $T_e$ of $G(x,\cdot)$ and
  $G(y,\cdot)$ are both in the one-dimensional $\relnullker{e}$.
\end{proof}

%%%%%%%%%%%%%%%%%%%%%%%%%%%%%%%%%%%%%%%%%%%%%%%%%%%%%%%%%%%%
\section{Green Functions with finite periodic ratios}
%%%%%%%%%%%%%%%%%%%%%%%%%%%%%%%%%%%%%%%%%%%%%%%%%%%%%%%%%%%%

Let $\pi\from T\to H$ and $M$ be as in Theorem~\ref{T:main}.
\begin{defn}
  A {\em {\finiteRS} for $M$} is a {\generalizedRS}, $r$, whose values,
$r_e$, are never $0$ or $\infty$; in other words, \eqref{eq:recursion}
always holds.
\end{defn}

As a step toward proving Theorem~\ref{T:main}, we first discuss the
simpler case where all ratios are finite.

\begin{thm}\label{T:main_simple}
  The following are equivalent:
  \begin{enumerate}
  \item $M$ is invertible and $G(u,u) \ne 0$ for all $u$,
  \item there is a {\finiteRS} for $M$ with $\alpha(r)<1$.
  \end{enumerate}
\end{thm}

We start the proof of \thmref{main_simple} with the assumption that $M$ is
invertible and $G(u,u)\ne 0$ for all $u$, and prove that there is a ratio
system for $M$ with $\alpha(r)<1$. Indeed, for any edge $e_0\in H$, take
some edge $e=(u,v)$ of type $e_0$ and define
\[
r_{e_0} = \frac{G(u,v)}{G(u,u)}.
\]
For a path $\vec p = \{e_i\}_{i=0}^n$ in $H$, it will be convenient to
denote $r(\vec p) = \prod r_{e_i}$, i.e., the product of $r_e$ over $\vec
p$, with repeated edges taken into account.

For an edge $e$, we consider the function $f$ on $T_e$ defined by $f =
\frac{G(u,\cdot)}{G(u,u)}$. As noted above, $f\in\relker{e}$. The following
relates $f$ to the {\finiteRS}.

\begin{claim}\label{C:prod_r}
  Consider a non-backtracking path $\vec p$ in $H$ beginning with $e_0$.
  Lift the path to $\tH$, starting at $u$ and let $y$ be the terminal
  vertex. Then $f(y)=r(\vec p)$.
\end{claim}

\begin{proof}
  We induct on the path's length. If $n=0$ the claim states $f(u)=1$, which
  is true. If $n=1$, the claim states $f(v)=r_e$ which is also true by the
  choice of $r$. For the inductive step, let $e'=(x,y)$ be an edge of type
  $e_k$ in $T_e$, with $x$ closer to the root, then it suffices to show
  that $f(y) = r_{e_k} f(x)$ (this is why we call $r$ a {\finiteRS}).

  Since the restriction of $f$ to $T_{e'}$ is in $\relker{e'}$, by
  \lemref{one_dim} it is proportional to $G(x,\cdot)$, and so
  \[
  \frac{f(y)}{f(x)} = \frac{G(x,y)}{G(x,x)} = r_{e_k}. \qedhere
  \]
\end{proof}

To check that $r$ is indeed a {\finiteRS}, note that $(Mf)(v)=0$. However,
\[
(Mf)(v) = m_{vv} r_e + m_{e^{-1}} + \sum_{e\to e'} m_{e'} r_e r_{e'},
\]
which is just \eqref{eq:recursion}.

It remains to show that $\alpha(r)<1$. Consider $S_n = \sum_{\rho(x,u)=n}
|f(x)|^2$, i.e., the contribution to $\|f(x)\|^2$ from all $x$ at level $n$
of $T_e$. The matrix $R$ is constructed so that $S_n$ is the square of the
$l^2$-norm of the row corresponding to $e_0$ in $R^n$ (this is proved by
induction using \clmref{prod_r}).

By Perron-Frobenius theory, there is an $e_0$ and $c>0$ such that $S_n \ge
c\alpha(r)^n$ for $n$ large (see e.g., \cite[Theorem 8.1.26]{HJ85}). Since
$\sum_n S_n = \|f_n\|^2$ is finite, it must be that $\alpha(r)<1$.

\medskip

We now prove the second direction of \thmref{main_simple}. Given a ratio
system $r$ with $\alpha(r)<1$, we will first construct the Green function
for $M$, and then use it to construct the inverse of $M$. For a vertex
$u\in T$ we define $f_u(v) = r(\vec p(u,v))$, where $\vec p(u,v)$ is (the
projection to $H$ of) the unique path from $u$ to $v$. By convention,
$f_u(u)=1$. Note that $f_u\in l^2$, since for all $\eps>0$ we have $S_n =
\sum_{\rho(x,u)=n} |f(x)|^2 = (\alpha(r)+\eps)^n$ for large $n$, which
converges if $\eps< 1-\alpha(r)$.

\begin{claim}\label{C:vanish_at_root}
  We have $(M f_u)(v)=0$ whenever $v\ne u$. If $e$ is (the type of) an edge
  incident to $u$ then $(Mf_u)(u)=0$ iff $r_e r_{e^{-1}}=1$. In particular,
  either $r_e r_{e^{-1}}=1$ for all edges, or for no edges.
\end{claim}

\begin{proof}
  When $v\ne u$ the claim is immediate from \eqref{eq:recursion}.

  Let $e$ be an edges originating at $u$. We have $(Mf_u)(u) = m_{uu} +
  \sum_{e'=(u,v)} m_{e'} r_{e'}$, where the sum is over all edges
  originating at $u$. Multiplying this by $r_{e^{-1}}$ and subtracting
  \eqref{eq:recursion} for $e^{-1}$ gives
  \[
  r_{e^{-1}} (Mf_u)(u) = m_e(r_e r_{e^{-1}} - 1).
  \]

  If $r_e r_{e^{-1}}=1$, then $r_{e^{-1}}\ne 0$ and therefor $(Mf_u)(u)$
  must be 0. If $r_e r_{e^{-1}}\ne 1$, then (since we assume $m_e\ne 0$)
  $(Mf_u)(u)$ cannot be 0.

  It follows that if $e,e'$ are two edges incident on $u$, then $r_e
  r_{e^{-1}}$ and $r_{e'} r_{e'^{-1}}$ are either both 1 or both different
  from 1. Since $H$ is connected, the same holds for any two edges.
\end{proof}

We now show that $r_e r_{e^{-1}}\ne 1$ for all $e$. Assume the contrary.
Take some non-backtracking cycle $\vec p\in H$, then also $r(\vec p) r(\vec
p^{-1})=1$, and so at least one of the two has absolute value at least one.
Without loss of generality suppose $|r(\vec p)|\ge 1$. However, in this
case $f_u$ cannot be in $l^2$, since $|f_u|\ge 1$ at all endpoints of the
path that is the lifting of $\vec p$ repeated any number of times.

Thus we find that $(Mf_u)(u)\ne 0$ for all $u$. Consequently, there is a
function $g_u$ proportional to $f_u$ such that $(M g_u)=\delta_u$. Let us
define $\tG(u,v)=g_u(v)$. Our next goal is to prove that $\tG$ is the Green
function for $M$. Let us establish some of its properties \footnote{Note
  that since $M$ is a symmetric operator, we expect $\tG$ to be symmetric
  as well. While this fact has a direct proof, the following lemma is
  easier and suffices for our purposes. The symmetry of $\tG$ will
  eventually follow by proving that $\tG$ is the Green function of $M$.}.

First, note that there are constants $K_1,K_2$ such that for any $x,z$ and
any $y$ on the path $[x,z]$
\begin{equation}\label{eq:split}
  K_1 \left| \tG(x,z) \right|
  \le \left| \tG(x,y)\tG(y,z) \right|
  \le K_2 \left| \tG(x,z) \right|.
\end{equation}
This follows from the definition of $\tG$ in terms of the {\finiteRS}, $r$,
and the fact that $\tG(y,y)$ takes on only finitely many values as it
depends only on $y$'s type.

The second property is that $\tG(x,\cdot)$ and $\tG(\cdot,x)$ are in $l^2$.
Moreover, multiplying a row or a column of $\tG$ by a ``slowly increasing
function'' (like that in the lemma that follows) still yields an $l^2$
function. Recall that we denote the distance between vertices $x,y$ in $T$
by $\rho(x,y)$.

\begin{claim}\label{C:columns-in-l2}
  For any vertex $u$, let $h_1 = \rho(u,\cdot) \tG(u,\cdot)$ and $h_2 =
  \rho(u,\cdot) \tG(\cdot,u)$. There is some constant $c$ independent of
  $u$ such that $\|h_1\|_2<c$ and $\|h_2\|_2<c$.
\end{claim}

\begin{proof}
  Since there are finitely many types for $u$, it suffices to consider a
  single $u$. As before, let $S_n=\sum_{\rho(x,u)=n} |\tG(u,x)|^2$, so that
  for any $\eps>0$ we have $S_n = (\alpha(r)+\eps)^n$ for large $n$. It
  follows that
  \[
  \|h_1\|^2 = \sum_n n^2 S_n < \infty.
  \]
  For $h_2$, consider $S_n' = \sum_{\rho(x,u)=n} |\tG(x,u)|^2$. This equals
  the $l^2$ norm of a column of $R^n$, and therefore has the same
  asymptotics as $S_n$.
\end{proof}

The next two claims imply that $M$ is invertible.

\begin{claim}\label{C:conv}
  For every $h\in l^2(T)$ there exists $\theta \in l^2(T)$ such that $M
  \theta = h$.
\end{claim}

\begin{claim}\label{C:kernel}
  The kernel of $M$ is trivial.
\end{claim}

\begin{proof}[Proof of \clmref{conv}]
  Given $h\in l^2(T)$, define $\theta$ as
  \[
  \theta(v)=\sum_{u\in V(T)} \tG(u,v) h(u);
  \]
  \clmref{columns-in-l2} and Cauchy-Schwarz combined imply that the above
  sum converges. It also follows that $M \theta = h$ at each vertex. It
  remains to show that $\theta\in l^2(T)$, so we need to bound
  \begin{align}
    \begin{split}\label{eq:theta_bd}
      \sum_v \big| \theta(v) \big|^2
      &\le \sum_{u,v,w} \left|h(u)h(w)\tG(u,v)\tG(w,v) \right| \\
      &= \sum_{u,w} \big| h(u)h(w)H(u,w) \big|,
    \end{split} \\
    \intertext{where}
    H(u,w) &= \sum_v \Big| \tG(u,v)\tG(w,v) \Big|.   \notag
  \end{align}

  The tree paths $[v,u]$ and $[v,w]$ separate at some point $x$, which is
  along the path $[u,w]$ (possibly, $x=u$ or $x=w$.) Separating the last
  sum according to the value of $x=x(v,u,w)$, using \eqref{eq:split} and
  summing over $v$ we get
  \begin{equation}\label{eq:H_bd}
  \begin{aligned}
  H(u,w)
  & \le c \sum_{\Tuxwv}  \big| \tG(u,x)\tG(w,x)\tG(x,v)^2 \big|  \\
  & \le c' \sum_{\Tuxw}  \big| \tG(u,x)\tG(w,x) \big|, \\
  \end{aligned}
  \end{equation}
  where the first diagram means that we sum over all vertices $u,w,v,x$
  such that the subtree they span has the given structure (with $x$ at the
  branching point in the first case, and along the path in the second).
  Note that it is possible to use \eqref{eq:split} at this stage to
  eliminate $x$. We keep $x$ to facilitate discussion in the next section.
  Inserting \eqref{eq:H_bd} in \eqref{eq:theta_bd} gives
  \[
  \sum_v |\theta(v)|^2 \le c \sum_{\Tuxw} |h(u)h(w)\tG(u,x)\tG(w,x)|.
  \]
  Fix an arbitrary root in $V(T)$ and let $z$ be the last common ancestor
  of $u,w$ in the above sum. Either $x\in[u,z]$ or $x\in[w,z]$. We bound
  the sum for $x\in[u,z]$, the proof for other sum being identical. We need
  to show that
  \[
  \sum_{\Tuxzw} \Big|h(u)h(w)\tG(u,x)\tG(w,x)\Big| < \infty.
  \]
  Note that $z$ must be the common ancestor of $x,w$. To avoid division by
  zero, denote for any two vertices $u,v$
  \[
  \rho'(u,v) = \max\{1,\rho(u,v)\}. 
  \]

  Using \eqref{eq:split} for $z\in[x,w]$ and Cauchy-Schwarz we have
  \begin{align*}
    \Big|h(u)h(w)\tG(u,x)\tG(w,x)\Big|
    &\le c \Big| h(u)h(w)\tG(u,x)\tG(z,x)\tG(w,z) \Big| \\
    &\le c \Big| h(w)\tG(u,x)\tG(z,x)
                 \frac{\rho'(z,x)\rho'(x,u)}{\rho'(z,w)} \Big|^2 \\
    & \hspace{3cm} + c \Big| h(u)\tG(w,z)
                 \frac{\rho'(z,w)}{\rho'(z,x)\rho'(x,u)} \Big|^2.
  \end{align*}
  Consider the first term. Its contribution to the sum can be written as
  \[
  c \sum_{\Tuxzw} \Big| h(w)\tG(u,x)\tG(z,x)
                 \frac{\rho'(z,x)\rho'(x,u)}{\rho'(z,w)} \Big|^2.
  \]
  $u$ lies below $x$ in the tree, and the sum of $\rho'(x,u)^2
  |\tG(u,x)|^2$ over such $u$ is uniformly bounded. Similarly we can sum
  $\rho'(z,x)^2 |\tG(z,x)|^2$ over all $x$ below $z$. Summing over $u$ and
  then over $x$ we are left with
  \[
  c \sum_{\Tzw} \frac{|h(w)|^2}{\rho'(z,w)^2}.
  \]
  Since $z$ must lie on the path from $w$ to the root, summing over $z$, we
  are left with
  \[
  c\sum_{n=0}^\infty \frac{|h(w)|^2}{(\max(1,n))^2} 
  \leq c' \|h\|_2^2 < \infty.
  \]

  The second term above is similarly bounded, except that we sum first over
  $w$, then over $z$, $x$ and $u$ in that order to end up with
  $c'\|h\|_2^2$ again.
\end{proof}

\begin{proof}[Proof of \clmref{kernel}]
  Suppose that $Mf=0$. Then for every function $g \in l^2(T)$, using the
  symmetry of $M$,
  \[
  0 = \langle M f, g \rangle
  = \langle f, M^* g\rangle
  = \langle f, \overline{M} g\rangle.
  \]
  Since we have a {\finiteRS} for $M$ with $\alpha<1$, we also have one for
  $\overline{M}$. By \clmref{conv} for $\overline M$, its image is all
  of $l^2$, and so for every $h \in l^2(T)$ we have $0=\langle f,
  h\rangle$. Thus $f$ must be 0.
\end{proof}

This completes the proof of \thmref{main_simple}. Now we give a proof of
Theorem~\ref{T:alpha1general} in the case of finite ratio systems.

\begin{thm}\label{T:alpha_1}
  If there is a {\finiteRS} for $M$ with $\alpha(r)=1$ then $M$ is
  not invertible.
\end{thm}

\begin{proof}%[Proof of \lemref{alpha_1}]
  Fix some edge $e=(u,v) \in \tH$. Using a {\finiteRS} $r$ with
  $\alpha(r)=1$ we consider the function
  \[
  f_n(x) = \begin{cases}
    r(\vec p) & \mbox{if $x\in T_e$ and $\rho(x,u)=n$,} \\
    0 & \text{otherwise,}
  \end{cases}
  \]
  and let $F_n = \sum_{k=0}^n f_k$. Thus $\|F_n\|^2=\sum^n \|f_k\|^2$.
  Since $MF_n$ is 0 on levels $1,\dots,n-1$ of the tree, it is supported on
  level 0 and levels $n,n+1$. Consequently $\|MF_n\|^2 = O(\|f_n\|^2)$.

  Recall that $\|f_n\|^2$ is the sum of the row corresponding to $e$ in
  $R^n$. Since the Perron eigenvalue of $R$ is 1, this sum grows
  polynomially in $n$ and we have $\|f_n\|^2 = P(n)+o(1)$ for some
  polynomial $P$. By choosing the edge $e$ to be in the support of the
  Perron eigenfunction of $R$ we can guarantee $P\not\equiv0$.

  Assume $\deg P=k$ (which might be 0), with leading term $c n^k$. Then
  $\|F_n\|^2 = c n^{k+1}/(k+1)+ O(n^k)$, while $\|M F_n\| = O(n^k)$. Thus
  $\|MF_n\| / \|F_n\| \to 0$ as $n\to\infty$, and so $M$ cannot have a
  bounded inverse.
\end{proof}

\begin{coro}\label{C:spectral_radius}
  For any finite graph $H$, we have $\rho(\tB)=\sqrt\gr$. Moreover, $\{\lam
  : |\lam|=\sqrt\gr\} \subset \sigma(\tB)$.
\end{coro}

\begin{proof}
  By \thmref{small_spec}, it suffices to show that $M=Q_\lam$ is invertible
  for $|\lam|>\sqrt\gr$, and is not invertible for $|\lam|=\sqrt\gr$. Since
  $m_{vv} = d_v-1+\lam^2$ and $m_e=-\lam$ for any $e$, it is
  straightforward to see that $r_e\equiv \lam^{-1}$ satisfies
  \eqref{eq:recursion}. This {\finiteRS} has $\alpha(r)=\lam^{-2}\gr$.
  \thmref{main_simple} shows that $Q_\lam$ is invertible whenever $|\lam| >
  \sqrt\gr$. However, if $|\lam|=\sqrt\gr$ then $\alpha(r)=1$ and
  \thmref{alpha_1} implies that $Q_\lam$ is singular.
\end{proof}

%%%%%%%%%%%%%%%%%%%%%%%%%%%%%%%%%%%%%%%%%%%%%%%%%%%%%%%%%%%%
\section{Zeroes and Infinities}
%%%%%%%%%%%%%%%%%%%%%%%%%%%%%%%%%%%%%%%%%%%%%%%%%%%%%%%%%%%%
\label{se:maintheoremproof}

In this section we extend the results of the previous section to the case
where $G(u,u)$ could vanishes for some $u$, i.e., where ratios are allowed
to be zero or infinity; we finish the proofs of Theorem~\ref{T:main} and
\ref{T:small_spec}. We discuss some examples where $G(u,u)$ vanishes for
some $u$ in Section~\ref{se:examples}.

We start with some remarks about
Definition~\ref{D:generalized-ratio-system}. Conditions~(b) and (c) replace
condition~(a) in the situation of zero and infinite ratios. Condition (b)
is easy to justify: $m_{e^{-1}}/r_e=\infty$, so there must be an infinite
term in the sum as well. Condition (c) effectively states that the only way
to have $r_f=\infty$ is as in condition (b). In particular, there cannot be
two edges leaving a vertex with infinite ratios (since then their inverses
both must and cannot have ratio zero).

As before, we need to extend $r$ to non-backtracking paths, and would like
$r(\vec p)$ to be the product of the ratios along the edges of the path. In
the presence of infinite ratios this needs clarification. We define $r(\vec
p)$ inductively. For paths of a single edge, $r(e)=r_e$, and for the an
empty path, $\eps$, $r(\eps)=1$).

For longer paths, note that an infinite ratio must be preceded by a zero
ratio, so we must determine the value of $0\cdot\infty$. For paths of
length $m\ge2$ we define $r(\vec p)$ inductively as follows:
\begin{equation}\label{eq:r_path}
  r(e_1,\dots,e_m) = \begin{cases}
    r_{e_1}r(e_2,\dots,e_m) & r_{e_1},r_{e_2} \ne \infty, \\
    -\frac{m_{e_2}}{m_{e_1^{-1}}} r(e_3,\dots,e_m) & 
    r_{e_1}=0, r_{e_2}=\infty, \\
    \infty & r_{e_1} = \infty. \\
  \end{cases}
\end{equation}
In the case $r_{e_1}=0$ and $r_{e_2}=\infty$, we effectively define
``$r_{e_1}r_{e_2}$'' (literally zero times infinity) to be
$r(e_1,e_2)=-m_{e_2}/m_{e_1^{-1}}$, since the uniqueness of $G$ implies
that $G(v,v)=0$ and $G(v,u)\ne 0$ implies that $G(v,w)\ne 0$ for exactly
one other neighbour of $u$, and furthermore $m_{vw}G(v,w)+m_{vu}G(v,u)=0$.

We remark that with this definition of $r(\vec p)$, the proof of
Theorem~\ref{T:alpha_1} carries over to prove
Theorem~\ref{T:alpha1general}.

We first assume that $M$ is invertible, and find the promised ratio system
$\{r_e\}$. Let $G$ be the Green function. Let $e=(u,v)$ be some edge in $T$
of type $e_0$. We define $r_e$ to be the generalized ratio $G(u,v)/G(u,u)$
as before, provided $G(u,v)$ and $G(u,u)$ do not both vanish. Otherwise,
since $MG(u,\cdot)$ does not vanish at $u$, there must be a neighbour $x$
of $u$ with $G(u,x)\ne0$. We then define $r_e=G(x,v)/G(x,u)$. In short
there is some $x$, either $x=u$ or $x\notin T_e$ such that $G(u,x),G(x,v)$
are not both zero and
\[
r_e = G(x,v)/G(x,u).
\]
By \corref{proportional}, this ratio does not depend on our choice of $x$.
As before, $r_e$ depends only on the type of $e$, so this it is naturally
equivalent to a function on edges of $H$.

\begin{claim}
  The function $r$ defined above is a {\generalizedRS} for $M$ with
  $\alpha(r)<1$.
\end{claim}

\begin{proof}
  Throughout we assume that $e=(u,v)$ and that $x$ is either $u$ or outside
  $T_e$ and $G(x,v)/G(x,u)$ is defined (and therefore equals $r_e$). In
  particular $x\ne v$ and so $MG(x,\cdot)=\delta_x$ vanishes at $v$. Thus
  we have
  \begin{equation} \label{eq:MGx_v}
    0 = m_{vv} G(x,v) + m_{e^{-1}} G(x,u)
      + \sum_{\substack{w\sim v\\w\ne u}} m_{vw} G(x,w).
  \end{equation}

  To check (a), note that $r_e\ne0$ implies $G(x,v)\ne0$, and so the same
  $x$ may be used to find $r_f$ for $f=(v,w)$. Writing
  $G(x,w)=G(x,v)r_{vw}$ in \eqref{eq:MGx_v} gives
  \[
  0 = m_{vv} + m_{e^{-1}} \frac{G(x,u)}{G(x,v)} +
    \sum_{\substack{w\sim v\\w\ne u}} m_{vw} r_{vw}.
  \]
  This is just \eqref{eq:recursion}, since $\frac{G(x,u)}{G(x,v)}=1/r_e$
  (meaning $0$ if $r_e=\infty$).

  To prove (b), suppose that $r_e=0$. Then we have $G(x,u)\ne0$ and
  $G(x,v)=0$, and so \eqref{eq:MGx_v} gives
  \[
  \sum_{\substack{w\sim v\\w\ne u}} m_{vw} G(x,w) = -m_{vu} G(x,u) \ne 0
  \]
  Thus not all of $G(x,w)$ are 0. Let $w\sim v$ such that $G(x,w)\ne0$,
  then the same $x$ can be used to find $r_f$ for $f=(v,w)$, and $r_{vw} =
  G(x,w)/G(x,v)=\infty$.

  To verify (c), note that if $r_f = \infty$ for $f=(v,w)$, then by
  \corref{proportional}, we must have $G(x,v)=0$ for all $x \not \in V(T_f)
  \setminus \{v\}$. In particular $G(v,v)=0$. It follows that $G(v,w) \neq
  0$, so $w$ can be used to find $r_{f^{-1}}$ and so
  $r_{f^{-1}}=\frac{G(w,v)}{G(w,w)}\ne 0$. Finally, to check that $r_e=0$,
  suppose $x$ is used to find $r_e$. Then $G(x,v)=0$, for otherwise $x$
  could be used to find $r_f$, and $r_f$ would be finite.

  \medskip

  It remains to show that $\alpha(r)<1$. As before, we show that the Perron
  eigenvalue of $R$ satisfies $\alpha(r)<1$ by proving that for all $e$
  with $r_e\ne0$ the sum of row $e$ in $R^k$ tends to zero with $k$. Pick
  an edge $\tilde e=(u,v)$ of type $e$ with $r_e\ne0$. As above, fix some
  vertex $x \notin T_{\tilde e}\setminus\{u\}$ satisfying $G(x,v)\ne 0$. We
  claim that
  \begin{equation}\label{eq:ratios_in_L2}
    \sum_{y\in T_{\tilde e}} |G(x,y)|^2
    = |G(x,u)|^2 + |G(x,v)|^2 \sum_f \sum_{k=0}^\infty (R^k)_{ef}
  \end{equation}
  This implies the required result, since the left-hand-side is finite. To
  see \eqref{eq:ratios_in_L2}, for every $y\in T_{\tilde e}$ except for
  $u$, the path $[v,y]$ in $T$ projects to some non-backtracking walk $\vec
  p_y$ in $H$. This gives a bijection between such $y$ and non-backtracking
  walks in $H$ that begin with an edge $e'$ such that $e\to e'$. We have
  $G(x,y)=G(x,v)r(\vec p_y)$.

  The definition of $R$ is such that
  \[
  (R^k)_{ef} = \sum |r_{\vec p}|^2,
  \]
  where the sum is over paths in $H$ that begin with an edge $e'$ such that
  $e \to e'$, end with edge $f$ and, for some $m$, contain $k+m$ steps with
  $m$ edges with $r_{e_i}=\infty$. The above identity follows.
\end{proof}

\medskip

Conversely, given a {\generalizedRS} $r$ with $\alpha(r)<1$, we construct a
candidate for the Green function $\tG(\cdot,\cdot)$ as in the proof of
\thmref{main_simple} and use it to construct $M^{-1}$. Our first goal is to
find for each $u$ a function $f_u$ so that $M f_u = c\delta_u$. We can then
define $\tG(u,v)=c^{-1}f_u(v)$.

Generally we take $f_u(x) = r(\vec p)$ where $\vec p$ is (the projection
of) the path from $u$ to $x$. However, if there is an edge $e$ leaving $u$
with $r_e=\infty$, then this results in some infinite values of $f_u$. In
this case, roughly speaking, we ``divide by $\infty$''. All finite values
will become 0, and the infinite ones will become finite. We shall now make
this more precise.

Let $V_\infty$ denote the set of all $u$ such that there is an edge
$e=(u,v)$ with $r_e=\infty$:
\[
V_\infty = \{u : \exists v, r_{uv}=\infty\}
\]
As noted above, for $u\in V_\infty$, the $v$ with $r_{uv}=\infty$ must be
unique. Define $u^*$ to be this neighbour, $v$; otherwise, for $u\notin
V_\infty$, let $u^*=u$. Then (recall $\pi\from T\to H$ is the covering map)
set
\begin{equation}\label{eq:fu}
f_u(x) = \begin{cases}
  r({\vec p}) & \text{if $u\notin V_\infty$, where $\vec p = \pi([u,x])$,}
  \\
  r({\vec q}) & \text{if $u\in V_\infty$ and $x\in T_{(u,u^*)}
                      \setminus\{u\}$, where $\vec q = \pi([u^*,x])$}, \\
  0 & \text{otherwise.}
\end{cases}
\end{equation}

We further denote $Z(u) = \sum_{e=(u,v)} m_e r_e$ where the sum is over
edges leaving $u$ with $r_e\ne\infty$ (i.e., if one of the edges has
$r_e=\infty$, then it is not included).

\begin{claim} \label{C:factor}
  For any $u$ we have
  \begin{equation} \label{eq:factor}
    (M f_u)(u) = \begin{cases}
      m_{uu^*} & u\in V_\infty, \\
      m_{uu} + Z(u) & u \not\in V_\infty.
    \end{cases}
  \end{equation}
\end{claim}

\begin{proof}
  If $u \in V_\infty$ then we have $r_{uu^*}=\infty$. Consequently $f_u$ is
  supported on $T_{uu^*}$, $f_u(u)=0$ and for $v\sim u$ we have $f_u(v)$ is
  $1$ or $0$ according to whether or not $v=u^*$. The claim follows. For $u
  \notin V_\infty$ we have $f_u(u)=1$ and $f_u(v)=r_{uv}$ for all
  neighbours $v$ of $u$, so the equality holds also in this case.
\end{proof}

Define $V_0 = \{u : m_{uu}+ Z(u)=0\}$.

\begin{claim}
  \label{C:T0connected}
  If an edge $e=(u,v)$ has $r_e\notin\{0,\infty\}$ and $v\in V_0$, then
  $r_e r_{e^{-1}} = 1$ and $u\in V_0$.
\end{claim}

\begin{proof}
  \eqref{eq:recursion} can be written as $m_{vv} + m_{e^{-1}}/r_e + Z(v) -
  m_{e^{-1}}r_{e^{-1}}=0$. If $v\in V_0$ this becomes $1/r_e-r_{e^{-1}}=0$,
  or $r_e r_{e^{-1}} = 1$. Applying this to \eqref{eq:recursion} for
  $e^{-1}$ gives $m_{uu}+Z(u)=0$, so $u\in V_0$ as well.
\end{proof}

\begin{claim}
  \label{C:double_infinity}
  If $u\in V_0$, some edge entering $u$ has ratio in $\{0,\infty\}$, then
  one edge $e$ entering $u$ has $r_e=r_{e^{-1}}=\infty$, and all other
  edges entering $u$ have ratio 0.
\end{claim}

\begin{proof}
  If $r_{vu}=0$ for some $v$, then by (b) there is an $e=(w,u)$ with
  $r_{e^{-1}}=\infty$. This implies by (c) that all edges $e'$ entering $u$
  except for $e$ have $r_{e'}=0$. By (c), $r_e\ne0$. Applying (a) to $e$
  reduces to $m_{uu} + m_{e^{-1}}/r_e + Z(u)=0$. Since $u\in V_0$, we must
  have $r_e=\infty$.

  If $r_e=\infty$ for some $e=(v,u)$, then we claim $r_{e^{-1}}=\infty$: if
  it were finite, applying (a) to $e$ would give $m_{uu} + Z(u) -
  m_{e^{-1}}r_{e^{-1}}=0$. This implies $r_{e^{-1}}=0$, contradicting (c).
  Since $r_e=r_{e^{-1}}=\infty$, again by (c) all edges $e'$ entering $u$
  except for $e$ have $r_{e'}=0$.
\end{proof}

\begin{claim}
  $V_0\subset V_\infty$.
\end{claim}

\begin{proof}
  Suppose that $u\in V_0\setminus V_\infty$. Since there are no edges with
  infinite ratio leaving $u$, by \clmref{double_infinity} all edges
  entering $u$ have finite nonzero ratios. By \clmref{T0connected} this
  implies that all of $u$'s neighbours are in $V_0$, and by
  \clmref{double_infinity} they are all in $V_0\setminus V_\infty$.

  It follows that if $V_0\setminus V_\infty$ is nonempty then it must be
  the entire tree. If $V_0\setminus V_\infty = T$ then for all edges $r_e
  r_{e^{-1}}=1$. Then (as in \clmref{vanish_at_root}) taking a
  non-backtracking cycle, $\vec p$, in $H$ gives $r(\vec p) r(\vec p^{-1})
  = 1$, and so $\alpha\ge 1$ --- a contradiction.
\end{proof}

\begin{claim}
  For any $u$, $(Mf_u)(u) \ne 0$.
\end{claim}

\begin{proof}
  By \clmref{factor}, either $u\in V_\infty$ and then $(Mf_u)(u)
  =m_{uu^*}\ne0$, or $u\notin V_\infty$. In the latter case $u\notin V_0$,
  and therefore $(Mf_u)(u) = m_{uu}+Z(u) \ne 0$.
\end{proof}

It follows that we may define
\begin{equation}\label{eq:tG}
  \tG(u,x) = \frac{f_u(x)}{(M f_u)(u)},
\end{equation}
so that $M \tG(u,\cdot)=\delta_u$. It remains to prove the generalized form
of \clmref{conv}:

\begin{claim}
  Given a {\generalizedRS} with $\alpha(r)<1$, for every $h\in l^2(T)$
  there exists a $\theta\in l^2(T)$ such that $M\theta=h$.
\end{claim}

The proof is very similar to the proof of \clmref{conv}, and we only point
out the differences.

\begin{proof}
  Equation~(\ref{eq:split}) only fails if $G(y,y)=0$. This problem is
  easily overcome by shifting the splitting point from $y$ to $y^*$, which
  is the unique neighbour of $y$ satisfying $r_{yy^*}=\infty$. When
  $G(y,y)\ne 0$ we define $y^*$ to be $y$. We have for suitable $K_1,K_2$
  and any $y\in[x,z]$
  \begin{equation}\label{eq:split_gen} \tag{3$'$}
    K_1 | \tG(x,z) |
    \le | \tG(x,y^*)\tG(y^*,z) |
    \le K_2 | \tG(x,z) |.
  \end{equation}
  This follows from the definition of $\tG$ and the fact that $y$ has only
  finitely many types.

  As in the preceding section, we define $H(u,w) = \sum_v \left|
    \tG(u,v)\tG(w,v) \right|$. Using \eqref{eq:split_gen} gives the bound
  \[
  H(u,w) \le c \sum_{\Tuxw} |\tG(u,x^*)\tG(x^*,w)|.
  \]
  In all subsequent estimates $x^*$ and $z^*$, respectively, take the place
  of $x$ and $z$ in all instances of $\tG$, but $x$ and $z$ remain
  unchanged in the ``tree sum'' notation and in $\rho'$.

  The only difference is that in place of \clmref{columns-in-l2} we need
  bounds of the form
\begin{equation}\label{eq:tGbound}
\sum_x \rho(u,x)^2 \left( |\tG(u,x^*)|^2 + |\tG(x^*,u)|^2\right) < c.
\end{equation}
Such bounds are an easy consequence of \clmref{columns-in-l2}: The map
$x\mapsto x^*$ maps no more than $\dmax+1$ $x$'s to each $x^*$. Since also
$|\rho(x^*,y^*)-\rho(x,y)|\le 2$, replacing some of the variables by their
starred versions introduces at most a constant factor into the bounds at
each stage of the computation.
\end{proof}

%%%%%%%%%%%%%%%%%%%%%%%%%%%%%%%%%%%%%%%%%%%%%%%%%%%%%%%%%%%%
\section{Examples}\label{se:examples}
%%%%%%%%%%%%%%%%%%%%%%%%%%%%%%%%%%%%%%%%%%%%%%%%%%%%%%%%%%%%

We wish to analyze two examples to highlight the necessity of considering
{\generalizedRS}s. Each of these examples demonstrates a general point.
Then we give a third example, which seems to be the ``smallest interesting
example'' of a graph with a non-regular universal cover. We describe the
calculations needed to determine its non-backtracking spectrum. In each of
these examples, the finite graph is a simple graph. In the next section we
shall show how each of these can be viewed as a covering of an even smaller
graph by using self-loops, multiple edges, or ``passing'' to a pregraph.

\subsection{Example 1: Fig{\`a}-Talamanca and Steger} \label{se:steger}

\cite{FigSteg} contains various examples of graphs where $G(u,u)=0$ for
some $u$ when $M$ is the weighted adjacency matrix of the tree. Consider
$K_4$, the complete graph on four vertices, whose six edges are written as
the disjoint union of three perfect matchings, $E_1,E_2,E_3$. Let the edges
in $E_i$ have weight $p_i$, for some positive real $p_1,p_2,p_3$. Without
loss of generality $p_1\ge p_2\ge p_3$. By symmetry (i.e., using
Proposition~\ref{pr:symmetry}), for any ratio system $r$, the ratios $r_e$
of all directed edges of $E_i$ are equal. Call the three ratios
$r_1,r_2,r_3$, and consider the ratio system is $r_1=\infty$, $r_2=r_3=0$.
We have $r(e_k,e_1)=-p_k/p_1$; the condition $\alpha(r)<1$ amounts to
\[
(p_2/p_1)^2+(p_3/p_1)^2 < 1.
\]
Thus the adjacency matrix for the tree is invertible with $G(u,u)=0$ if
$p_1^2 > p_2^2 + p_3^2$. (In \cite[Lemma~1.4]{FigSteg} it is shown that if
$p_1^2\le p_2^2+p_3^2$ then the adjacency matrix is not invertible.)

It is of interest to note that the relative kernel with respect to an $E_1$
edge is only one-dimensional (when $p_1^2>p_2^2+p_3^2$).

\subsection{Example 2: $(2,3)$-bipartite graphs}

Consider the adjacency matrix of a $(2,3)$-bipartite graph, i.e., a
bipartite graph with vertex set $V_2\cup V_3$ such that each vertex in
$V_k$ is of degree $k$. Then consider $r_{32}=\infty$, $r_{23}=0$, where
$e_{23}$ is an edge from $V_2$ to $V_3$ and $e_{32}$ the opposite. This
fails to be a generalized ratio for the adjacency matrix only at
condition~(c) of the definition. Furthermore, if we tried to omit
condition~(c), then this ratio system would have $\alpha<1$: A vertex of
degree 3 in the universal cover has two children so
$r(e_{32},e_{23})=-1/2$. This gives rise to $L^2$-functions, $f$, with
$Mf=\delta_r$ for any vertex $r$ of degree 2 (we can take $f$ to have value
0 at all degree two vertices, and value $-(-2)^{-a}$ at all vertices at odd
distance $2a-1$ to $r$).

It is easy to see that the adjacency matrix is not invertible here; indeed,
consider a long path in the tree, and a function supported on the degree
two vertices of this path, alternating $\pm 1$; this function, $f$, has
$Mf$ bounded but $\|f\|$ unbounded as the path length tends to infinity.
This example shows that condition (c) in the definition of a
{\generalizedRS} cannot be omitted from Theorem~\ref{T:main}.

\subsection{Example 3: Case study of $P_{1,2,2}$}

We consider here non-backtracking walks on lifts of the graph $P_{1,2,2}$
consisting of three paths between a pair of vertices $u,u'$ of lengths
$1,2,2$. Thus $u,u'$ have degree 3 and there are two other vertices,
$v,v'$, of degree 2. This is the same as $K_4$ with an edge removed.

To find the spectrum of $B(\widetilde P_{1,2,2})$ we need to determine
invertibility of $M=Q_\lam$. Let us look for ratio functions $r_e$ that
respect all the graph's symmetries, i.e.\ $r_{uv}=r_{uv'}=\cdots$. There
are only three distinct ratios:
\[
  x = r_{uv}  \qquad  y = r_{uu'}  \qquad  z = r_{vu}.
\]
Thus $y$ appears on two directed edges, and $x,z$ on four each. The
equations we need to satisfy are
\begin{align*}
  (\lam^2+1) x + \lam (1+xz) =0,     \\
  (\lam^2+2) y + \lam (1+2xy) =0,    \\
  (\lam^2+2) z + \lam (1+xz+yz) =0.
\end{align*}
Using the first equation to eliminate $z$ and the second to eliminate $x$
yields:
% maple code:
% eq1 := (L^2+1)*x+L*(1+x*z);
% eq2 := (L^2+2)*y+L*(1+2*y*x);
% eq3 := (L^2+2)*z+L*(1+z*(x+y));
%
% eq_y := subs(z=solve(eq1=0,z),x=solve(eq2=0,x),eq3);
% eq_y := simplify(-eq_y  * 2*y*(L^2*y+2*y+L) / (y+1/L));
\[
  \left( y + \lam^{-1} \right) \left( (2\lam^4+2\lam^2+4) y^2
    + \lam^3(\lam^2+1) y - \lam^2(\lam^2+1) \right) = 0.
\]
This leads to the trivial solution $x=y=z=-\lam^{-1}$ and two non-trivial
solutions for $x,y,z$ (or to be precise, two branches of one solution). The
condition $\alpha(r)<1$ reduces to $x^2 z^2 (1+2y^2)<1$. A computer plot of
the $\lambda$ where none of the solutions has $\alpha(r)<1$ is shown in
Figure~\ref{fig-tree-spec} (the shaded region).

%%%%%%%%%%%%%%%%%%%%%%%%%%%%%%%%%%%%%%%%%%%%%%%%%%%%%%%%%%%%
\section{Self-loops and Pregraphs via Finite Tree Quotients}
%%%%%%%%%%%%%%%%%%%%%%%%%%%%%%%%%%%%%%%%%%%%%%%%%%%%%%%%%%%%
\label{se:self}\label{se:pregraph}

We shall show that Theorem~\ref{T:main} can be generalized to universal
covers of all of the following structures: graphs/pregraphs with multiple
edges, self-loops, boundaries (with boundary conditions), edge and vertex
weights, and edge lengths (see \cite{Fri93}). The common theme is that all
these structure admit some sort of ``universal cover,'' $T$ that is
sufficiently symmetric that a mild generalization of Theorem~\ref{T:main}
holds.

First, note that Definition~\ref{D:generalized-ratio-system} makes sense
for any graph homomorphism $\pi\from T \to H$, regardless of whether or not
$H$ is finite (and even whether or not $\pi$ is a covering map).

\begin{defn}
  A {\em tree quotient} is a tree, $T$, along with a group, $\cG$, acting
  on it; we denote this $T/\cG$. By $V(T/\cG)$ we mean the $\cG$ orbits of
  $V(T)$, and similarly for $E(T/\cG)$. The action of $\cG$ is cofinite if
  $V(T/\cG)$ and $E(T/\cG)$ are finite; alternatively we say $T/\cG$ is
  {\em finite}. By a symmetric local operator on $T/\cG$ we mean a
  symmetric local operator on $T$ that is invariant under $\cG$. By a
  {\generalizedRS} on $T$ we mean a {\generalizedRS} as in
  Definition~\ref{D:generalized-ratio-system} with $H$ there taken to be
  $T$; by a {\generalizedRS} on $T/\cG$ we mean a {\generalizedRS} on $T$
  that is invariant under $\cG$, i.e., $r$ can be viewed as a function
  $r:E(T/\cG)\to \C\cup\{\infty\}$.
\end{defn}

Any finite graph and pregraph as described in \cite{Fri93,FriTill}
(possibly with boundary, self-loops, multiple edges, vertex and edge
weights, and edge lengths) admits a universal cover, $\pi\from T\to G$; $T$
is a tree (a graph) that inherits vertex and edge weights and lengths from
$G$ (boundary edges may be discarded with the Laplacian or adjacency matrix
modified appropriately). By the universal property of $T$, for any two
vertices in $T$ above the same $G$ vertex, $T$ admits an automorphism from
taking one vertex to the other. It follows that if $\cG$ is the group of
all automorphisms of $T$ over $G$, then $T/\cG$ is finite and has at most
as many edges and vertices as $G$. Hence to generalize
Theorem~\ref{T:main} to all the above situations it suffices to consider
finite tree quotients $T/\cG$.

Consider a symmetric, local operator, $M$, on a tree quotient, $T/\cG$. The
proof of Theorem~\ref{T:main} shows that a Green's function, $G$, for $M$
gives rise to generalized periodic ratios on $T$;
Proposition~\ref{pr:symmetry} shows that generalized periodic ratios also
live on $T/\cG$. Furthermore, given the periodic ratios $\{r_e\}$ from $G$,
we can reconstruct $G$ via \eqref{eq:fu} and \eqref{eq:tG}.

\begin{defn}
  A {\em classifying condition} for a tree quotient, $T/\cG$, is a
  condition $C=C(r)$ or even $C=C(r,M)$ on generalized periodic ratios,
  $r$, such that $C(r,M)$ is satisfied iff $r$ arises from the Green
  function, $G$, of an invertible, symmetric, local operator, $M$, on
  $T/\cG$.
\end{defn}

Given a tree quotient, we wish to find a classifying condition that is
simple to check. Theorem~\ref{T:main} says that if $\pi\from T\to H$ is a
universal cover of a finite graph, $H$, and $\cG$ is the group of
automorphisms of $T$ over $H$, then ``$\alpha(r)<1$'' is a classifying
condition for $T/\cG$.

Our first task is to generalize the matrix $R$ that describes $\alpha$ on a
tree quotient, $T/\cG$. For $e,f\in E(T)$ with $r_e,r_f\ne 0$, we define
$R_{e,f}$ as in \eqref{eq:defineR}; for $a,b\in E(T/\cG)$ pick an $e$ with
$e\in a$ and set
\[
R_{a,b} = \sum_{f\in b} R_{e,f}.
\]

In this setting, our proof of Theorem~\ref{T:main} shows the following:

\pagebreak

\begin{thm}
  Let $C=C(r)$ be a condition on generalized periodic ratios of a tree
  quotient, $T/\cG$, such that
  \begin{enumerate} \renewcommand{\labelenumi}{(\alph{enumi})}
  \item if $r$ arises from a Green function, then $C(r)$ is satisfied,
  \item if $r$ satisfies $C(r)$, then $(Mf_u)(u)\ne 0$ for all $u$ (with
    $f$ as in \eqref{eq:fu}), and
  \item if $r$ satisfied $C(r)$ then $\tG$ is a bounded kernel, i.e.,
    the map
    \[
    h(\,\cdot\,)\mapsto \sum_{u\in V(T)} \tG(u,\,\cdot\,) h(u),
    \]
    is bounded in $L^2$.
  \end{enumerate}
  Then $C$ is a classifying condition for $T/\cG$. Furthermore,
  \begin{enumerate}
  \item $(Mf_u)(u)\ne 0$ for all $u$ as long as there can exist no sequence
    of distinct vertices, $v_0,v_1,\ldots$ in $T$ for which
    $|r([v_0v_i])|\ge 1$ for all $i\ge 1$, and
  \item any function $\tG\from V(T)\times V(T)\to \C$ is a bounded
    kernel provided that it satisfies the bounds \eqref{eq:split} and
    \eqref{eq:tGbound}.
  \end{enumerate}
\end{thm}

Now let $\alpha(r)$ be the spectral radius of the operator $R$ defined
above as an operator on $l^2(V(T/\cG))$. If $\cG$ acts cofinitely on $T$,
then $R$ is a finite real matrix with non-negative entries, and a periodic
ratio arising from a Green function must satisfy $\alpha(r)<1$ since the
Green function is in $l^2$, using the analogue of \eqref{eq:ratios_in_L2}.
Conversely, a generalized periodic ratio with $\alpha(r)<1$ yields that
\[
R+ 2R^2+ 3R^3+\cdots
\]
is a bounded operator, from which we conclude \eqref{eq:tGbound} and that
there cannot exist distinct $v_0,v_1,\ldots\in V(T)$ with $|r([v_0v_i])|\ge
1$; the cofinite action on $\cG$ on $T$ shows that $((Mf_u)(u))^{-1}$
and $r_e$ are uniformly bounded, yielding \eqref{eq:split}. We conclude the
following theorem.

\begin{thm}
  Let $\cG$ act cofinitely on $T$. The condition $\alpha(r)<1$, defined
  above, is a classifying condition for $T/\cG$.
\end{thm}

We do not know a classifying condition when $\cG$ does not act cofinitely
on $T$.

\medskip

We now demonstrate how the trees in the three examples in the previous
section have smaller quotients, provided one is willing to pass to graphs
or pregraphs that may have self-loops and multiple edges.

Example~1 in the previous section covers the graph consisting of one vertex
with three half-loops, $e_1,e_2,e_3$, of respective weights $p_1\ge p_2\ge
p_3$. This is how it is presented in \cite{FigSteg}; they allow any number
of half-loops with weights $p_1\ge\dots\ge p_r$. The covering tree's
adjacency matrix is invertible iff $p_1^2 > p_2^2+\cdots+p_r^2$. The $r$
half-loop example is covered by any $r$-regular graph whose edges can be
partitioned into $r$ perfect matchings.

Example~2 covers the pregraph with two vertices, no self-loops, with one
vertex of degree two, and one vertex of degree three.

Example~3 covers the graph $H=(V,E)$, where $V=\{u,v\}$, where there are
two edges joining $u$ and $v$ and one half-loop at $u$.

\section{Concluding Remarks and Further Questions}

We have given a semi-algebraic characterization of invertible symmetric
local operators on a finite tree quotient. It would be interesting to know
what happens when the tree quotient is not finite and/or when the operator
is not symmetric (with respect to $l^2$ of any measure on the tree
vertices). For example, the non-backtracking {\em random} walk selects at
each step a uniform neighbour of $X_n$ different from $X_{n-1}$
\cite{ABLS,OW}.

Our characterization can be used, at least in principle, to compute the
various spectra of trees, including adjacency and Laplacian, and, less
obviously, non-backtracking. This may shed light on the questions mentioned
in the introduction. One interesting question, especially for
two-dimensional non-backtracking spectrum, is to know if there is a natural
spectral measure there, analogous to the Wigner semicircle law, as done by
McKay for regular base graphs. Specifically, we conjecture that for any
finite graph $H$ there is a measure $\mu$ supported on the non-backtracking
spectrum of the universal cover $\tH$, such that the uniform measure on
$\sigma(H^n)$ converges in probability (in a suitable topology) to $\mu$,
where $H^n$ is a random uniform $n$-lift of $H$ (possibly with some
corrections, say around $\pm 1$).

The proof of Theorem~\ref{T:alpha1general} is based on explicitly
constructing approximate eigenfunctions when there is a ratio system with
$\alpha=1$. It would be nice to be able to find a sequence of approximate
eigenfunctions of $M-\lambda$ for general $\lambda$ is $M$'s spectrum.

Finally we believe that for $\lambda$ on the boundary of the spectrum of
$M$ (as before), $M-\lambda I$, has a {\generalizedRS}, $r$, with
$\alpha(r)=1$. In particular, if $\{\lambda_n\}$ is a sequence of complex
numbers with a finite limit, $\lambda$, and $M$ is as usual, and $r_n$ is a
ratio system for $M-\lambda_n$, then $r_n$ has a limit point, $r$, where we
compactify $\C\cup\{\infty\}$ as usual. Is $r$ necessarily a ratio system?
In particular is it impossible to have $r_e=r_f=\infty$ for two edges,
$e,f$, with the same tail? If this does happen, might it still be possible
to define $\alpha$ for the limit point, $r$?

\bibliographystyle{alpha}
\bibliography{NBRW}

\end{document}